\documentclass[11pt, a4paper,reqno,oneside]{amsart}
\usepackage[a4paper, total={6in, 9in}]{geometry}
\usepackage{subscript}
\usepackage{listings}
\usepackage[T1]{fontenc}

\usepackage{float}
\usepackage{amsfonts}
\usepackage{mathtools}
\mathtoolsset{showonlyrefs}
\numberwithin{equation}{section}
\usepackage{mathrsfs}  
\usepackage{amsthm}
\usepackage{array, boldline, makecell, booktabs}
\usepackage[final, pdftex, pdfpagelabels, pdfstartview={FitH},
            bookmarks, plainpages=false, linktoc=all,
            linkcolor=black, citecolor=blue, urlcolor=blue,
            filecolor=black]{hyperref}
\usepackage{cleveref}
\usepackage{amssymb}
\usepackage[USenglish]{babel}
\usepackage{csquotes}

\usepackage{enumerate}
\usepackage{xcolor}
\usepackage{fvextra}
\DefineVerbatimEnvironment{wrapverbatim}{Verbatim}{breaklines=true}

\usepackage{dynkin-diagrams}
\usepackage{standalone}

\usepackage{comment}

\usepackage{tikz}
\usetikzlibrary{calc}
\usepackage{subcaption}
\usepackage[title]{appendix}
\usepackage{oldgerm}
\usepackage{multicol}

\def\Z{\mathbb{Z}}

\usepackage[textsize=tiny]{todonotes}
\usepackage{minted}
\setminted{breaklines=true}

\usepackage{xparse}

\NewDocumentCommand{\NewTodoAuthor}{mmm}{%
  \expandafter\NewDocumentCommand\csname #1\endcsname{s m}{%
    \IfBooleanTF{##1}
      {\todo[inline,author=#2,color=#3]{##2}}%
      {\todo[author=#2,color=#3]{##2}}%
  }%
}

\NewTodoAuthor{JE}{JE}{cyan!20}
\NewTodoAuthor{NL}{NL}{orange!25}
\NewTodoAuthor{DP}{DP}{blue!20}
\NewTodoAuthor{JS}{JS}{green!25}
\NewTodoAuthor{AW}{AW}{violet!20}
\NewTodoAuthor{GW}{GW}{red!20}

\theoremstyle{plain}
\newtheorem{theorem}{Theorem}[section]
\newtheorem{lemma}[theorem]{Lemma}
\newtheorem{prop}[theorem]{Proposition}
\newtheorem{cor}[theorem]{Corollary}

\theoremstyle{definition}
\newtheorem{definition}[theorem]{Definition}
\newtheorem{remark}[theorem]{Remark}
\newtheorem{example}[theorem]{Example}

\def\C{\mathbb{C}}
\def\F{\mathbb{F}}

\newcommand{\dwd}[1]{\mathcal{D}_{#1}}
\newcommand{\comI}[1]{\mathcal{C}_{#1}}

\newcommand{\DrawGridAndPoints}[3]{%
  \begin{tikzpicture}[x=0.28cm,y=0.28cm,line cap=round]
    \def\n{16} 

    \draw[step=1,very thin,gray!40] (0,0) grid (\n,\n);
    \draw[line width=0.9pt] (0,0) rectangle (\n,\n);

    \foreach \x in {0,#1,...,\n} {\draw[line width=1.1pt] (\x,0) -- (\x,\n);}
    \foreach \y in {0,#2,...,\n} {\draw[line width=1.1pt] (0,\y) -- (\n,\y);}

    \foreach \x in {0,...,15} {
      \node[font=\tiny, anchor=south] at (\x+0.5,\n) { \x };
    }
    \foreach \y in {0,...,15} {
      \pgfmathtruncatemacro{\lab}{15-\y}
      \node[font=\tiny, anchor=east] at (0,\y+0.5) { \lab};
    }

    \foreach \i/\val in {
      0/0, 1/8, 2/4, 3/12, 4/2, 5/10, 6/6, 7/14,
      8/1, 9/9, 10/5, 11/13, 12/3, 13/11, 14/7, 15/15
    }{
      \fill[red] (\i+0.5,\val+0.5) circle (1.6pt);
    }
    \foreach \i/\val in {
      0/15, 1/7, 2/11, 3/3, 4/13, 5/5, 6/9, 7/1,
      8/14, 9/6, 10/10, 11/2, 12/12, 13/4, 14/8, 15/0
    }{
      \fill[blue] (\i+0.5,\val+0.5) circle (1.6pt);
    }

    \node[font=\footnotesize, anchor=north] at (\n/2,-1.2) {#3};
  \end{tikzpicture}%
}

\title{Bruhat intervals that are large hypercubes}

\author{Jordan Ellenberg}
\address{Department of Mathematics, University of Wisconsin-Madison, Madison, WI USA}
\email{ellenber@math.wisc.edu}

\author{Nicolas Libedinsky}
\address{Departamento de Matemáticas, Universidad de Chile, Chile}
\email{nlibedinsky@gmail.com}

\author{David Plaza}
\address{Instituto de Matemáticas, Universidad de Talca, Chile}
\email{dplaza@utalca.cl}

\author{Jos\'e Simental}
\address{Instituto de Matemáticas, Universidad Nacional Autónoma de México. México}
\email{simental@im.unam.mx}

\author{Geordie Williamson}
\address{School of Mathematics and Statistics, University of Sydney, Australia}
\email{g.williamson@sydney.edu.au}

\date{\today}

\begin{document}

\begin{abstract}

We study the question of finding big Bruhat intervals that are poset hypercubes in the symmetric group $S_n$.
Using permutations suggested by AlphaEvolve (an evolutionary coding agent developed by Google DeepMind), we were led to an unusual situation in which the agent produced a pattern which performed well for the $n$ tested, and which we show works well for general $n$.
When $n$ is a power of 2 we exhibit a hypercube of dimension $O(n\log n)$, matching the largest possible dimension up to a constant multiple.
Furthermore,  we give an exact characterization of the vertices of this hypercube: they are precisely the \emph{dyadically well-distributed} permutations—a simple digitwise property that already appeared in connection with Monte Carlo integration and mathematical finance. 
The maximal dimension of a  Bruhat interval that is an hypercube in $S_n$ gives a lower bound (and possibly is equal to)  the maximal possible coefficient of the second-highest degree term in the 
 Kazhdan--Lusztig $R$-polynomial in $S_n$. 
As a surprising consequence, we obtain a new lower bound of order $n\log n$ for the maximal number of frozen variables appearing in the cluster algebras attached to the open Richardson varieties in $S_n$, and a  similar result for moduli spaces of embeddings of Bruhat graphs. 

\end{abstract}

\maketitle

\section{Introduction}

Consider the symmetric group $S_n$ equipped with the Bruhat order.
We study the question of how large a hypercube (i.e.\ boolean) interval can
occur inside $S_n$. 
To the best of our knowledge, this problem was first posed in Reading's
thesis \cite{reading2002structure}, which established lower bounds for all
finite irreducible Coxeter groups. 
In the specific case of $S_n$, he showed that the largest hypercube has at
least rank
\begin{equation}
    n-1 + \left\lfloor \frac{n-2}{2} \right\rfloor .
\end{equation}
Further works that revolve around this problem include
\cite{tenner2007pattern, hultman2009pattern, tenner2022interval, lee2021toric}.

In this paper, we find an interval in $S_n$ which is poset isomorphic to a hypercube whose dimension is within a constant mutliple of the maximal possible size.

More precisely, our main theorem is the following (see \S~\ref{sec:dwd} for undefined terms):

\begin{theorem}
    Consider the symmetric group on $2^m$ elements, for $m \ge 1$. The set of dyadically well distributed permutations forms an interval in Bruhat order. This interval is poset isomorphic to a hypercube of dimension $m2^{m-1}$.
\end{theorem}

By Stirling's formula, any hypercube inside $S_{2^m}$ is of dimension bounded by $m2^m$.  Thus the dimension of our hypercube is within a constant factor of largest possible.

The interval referred to in Theorem~\ref{thm: main} consists of permutations which are ``dyadically well distributed,'' a number-theoretic condition related to binary expansion. We remark that the definition of dyadically well distributed permutations does not involve any Lie theoretic data so the fact that this set is indeed a Bruhat interval is quite striking. 
To the best of the authors' knowledge this is the first instance this phenomenon occurs in a non-trivial way.

\subsection{Background}
 The existence of natural and large hypercubes in Bruhat order is surprising. Anyone who experiments with Bruhat order will notice that the majority of small intervals are small hypercubes. It is also easy to see that the interval between the identity and any Coxeter element (for example, $s_1s_2...s_{n-1}$, where $s_i$ denotes a simple transposition) is a hypercube of dimension $n-1$. It is perhaps suprising that there exist hypercubes of dimension larger than $n-1$ (see Figure \ref{fig: S4 hypercube}). Moreover, as the intervals become larger it becomes harder and harder to find hypercubes (see Table \ref{fig: ratio Hyp}). This might lead one to guess that there are no hypercubes of superlinear dimension in Bruhat order. Our main result shows that the intuition gained from these small examples is far from the reality.

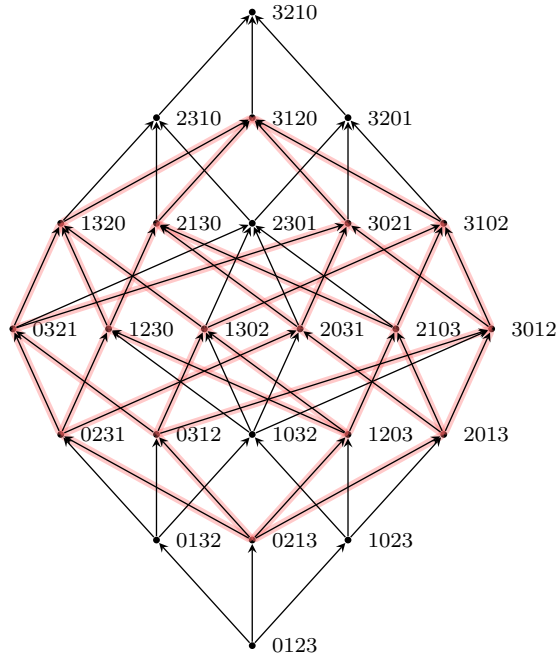
\begin{figure}[htbp]
    \centering
\begin{tikzpicture}[
    scale=.7,
    >=stealth,
    dot/.style={circle,fill=black,inner sep=0.9pt},
    every node/.style={font=\scriptsize},
    intervaledge/.style={line width=3pt, draw=red!60, opacity=0.35, line cap=round},
    arrow/.style={->, line width=0.5pt}
]


\node[dot] (v0123) at (0,0) {};
\node[anchor=west] at (0.18,0) {$0 1 2 3$};

\node[dot] (v0132) at (-1.8,2) {};
\node[anchor=west] at (-1.62,2) {$0 1 3 2$};

\node[dot] (v0213) at (0,2) {};
\node[anchor=west] at (0.18,2) {$0 2 1 3$};

\node[dot] (v1023) at (1.8,2) {};
\node[anchor=west] at (1.98,2) {$1 0 2 3$};

\node[dot] (v0231) at (-3.6,4) {};
\node[anchor=west] at (-3.42,4) {$0 2 3 1$};

\node[dot] (v0312) at (-1.8,4) {};
\node[anchor=west] at (-1.62,4) {$0 3 1 2$};

\node[dot] (v1032) at (0,4) {};
\node[anchor=west] at (0.18,4) {$1 0 3 2$};

\node[dot] (v1203) at (1.8,4) {};
\node[anchor=west] at (1.98,4) {$1 2 0 3$};

\node[dot] (v2013) at (3.6,4) {};
\node[anchor=west] at (3.78,4) {$2 0 1 3$};

\node[dot] (v0321) at (-4.5,6) {};
\node[anchor=west] at (-4.32,6) {$0 3 2 1$};

\node[dot] (v1230) at (-2.7,6) {};
\node[anchor=west] at (-2.52,6) {$1 2 3 0$};

\node[dot] (v1302) at (-0.9,6) {};
\node[anchor=west] at (-0.72,6) {$1 3 0 2$};

\node[dot] (v2031) at (0.9,6) {};
\node[anchor=west] at (1.08,6) {$2 0 3 1$};

\node[dot] (v2103) at (2.7,6) {};
\node[anchor=west] at (2.88,6) {$2 1 0 3$};

\node[dot] (v3012) at (4.5,6) {};
\node[anchor=west] at (4.68,6) {$3 0 1 2$};

\node[dot] (v1320) at (-3.6,8) {};
\node[anchor=west] at (-3.42,8) {$1320$};

\node[dot] (v2130) at (-1.8,8) {};
\node[anchor=west] at (-1.62,8) {$2 1 3 0$};

\node[dot] (v2301) at (0,8) {};
\node[anchor=west] at (0.18,8) {$2 3 0 1$};

\node[dot] (v3021) at (1.8,8) {};
\node[anchor=west] at (1.98,8) {$3 0 2 1$};

\node[dot] (v3102) at (3.6,8) {};
\node[anchor=west] at (3.78,8) {$3 1 0 2$};

\node[dot] (v2310) at (-1.8,10) {};
\node[anchor=west] at (-1.62,10) {$2 3 1 0$};

\node[dot] (v3120) at (0,10) {};
\node[anchor=west] at (0.18,10) {$3 1 2 0$};

\node[dot] (v3201) at (1.8,10) {};
\node[anchor=west] at (1.98,10) {$3 2 0 1$};

\node[dot] (v3210) at (0,12) {};
\node[anchor=west] at (0.18,12) {$3 2 1 0$};


\draw[intervaledge] (v0213) -- (v0231);
\draw[intervaledge] (v0213) -- (v0312);
\draw[intervaledge] (v0213) -- (v1203);
\draw[intervaledge] (v0213) -- (v2013);

\draw[intervaledge] (v0231) -- (v0321);
\draw[intervaledge] (v0231) -- (v1230);
\draw[intervaledge] (v0231) -- (v2031);

\draw[intervaledge] (v0312) -- (v0321);
\draw[intervaledge] (v0312) -- (v1302);
\draw[intervaledge] (v0312) -- (v3012);

\draw[intervaledge] (v1203) -- (v1230);
\draw[intervaledge] (v1203) -- (v1302);
\draw[intervaledge] (v1203) -- (v2103);

\draw[intervaledge] (v2013) -- (v2031);
\draw[intervaledge] (v2013) -- (v2103);
\draw[intervaledge] (v2013) -- (v3012);

\draw[intervaledge] (v0321) -- (v1320);
\draw[intervaledge] (v0321) -- (v3021);

\draw[intervaledge] (v1230) -- (v1320);
\draw[intervaledge] (v1230) -- (v2130);

\draw[intervaledge] (v1302) -- (v1320);
\draw[intervaledge] (v1302) -- (v3102);

\draw[intervaledge] (v2031) -- (v2130);
\draw[intervaledge] (v2031) -- (v3021);

\draw[intervaledge] (v2103) -- (v2130);
\draw[intervaledge] (v2103) -- (v3102);

\draw[intervaledge] (v3012) -- (v3021);
\draw[intervaledge] (v3012) -- (v3102);

\draw[intervaledge] (v1320) -- (v3120);
\draw[intervaledge] (v2130) -- (v3120);
\draw[intervaledge] (v3021) -- (v3120);
\draw[intervaledge] (v3102) -- (v3120);


\draw[arrow] (v0123) -- (v0132);
\draw[arrow] (v0123) -- (v0213);
\draw[arrow] (v0123) -- (v1023);

\draw[arrow] (v0132) -- (v0231);
\draw[arrow] (v0132) -- (v0312);
\draw[arrow] (v0132) -- (v1032);

\draw[arrow] (v0213) -- (v0231);
\draw[arrow] (v0213) -- (v0312);
\draw[arrow] (v0213) -- (v1203);
\draw[arrow] (v0213) -- (v2013);

\draw[arrow] (v1023) -- (v1032);
\draw[arrow] (v1023) -- (v1203);
\draw[arrow] (v1023) -- (v2013);

\draw[arrow] (v0231) -- (v0321);
\draw[arrow] (v0231) -- (v1230);
\draw[arrow] (v0231) -- (v2031);

\draw[arrow] (v0312) -- (v0321);
\draw[arrow] (v0312) -- (v1302);
\draw[arrow] (v0312) -- (v3012);

\draw[arrow] (v1032) -- (v1230);
\draw[arrow] (v1032) -- (v1302);
\draw[arrow] (v1032) -- (v2031);
\draw[arrow] (v1032) -- (v3012);

\draw[arrow] (v1203) -- (v1230);
\draw[arrow] (v1203) -- (v1302);
\draw[arrow] (v1203) -- (v2103);

\draw[arrow] (v2013) -- (v2031);
\draw[arrow] (v2013) -- (v2103);
\draw[arrow] (v2013) -- (v3012);

\draw[arrow] (v0321) -- (v1320);
\draw[arrow] (v0321) -- (v2301);
\draw[arrow] (v0321) -- (v3021);

\draw[arrow] (v1230) -- (v1320);
\draw[arrow] (v1230) -- (v2130);

\draw[arrow] (v1302) -- (v1320);
\draw[arrow] (v1302) -- (v2301);
\draw[arrow] (v1302) -- (v3102);

\draw[arrow] (v2031) -- (v2130);
\draw[arrow] (v2031) -- (v2301);
\draw[arrow] (v2031) -- (v3021);

\draw[arrow] (v2103) -- (v2130);
\draw[arrow] (v2103) -- (v2301);
\draw[arrow] (v2103) -- (v3102);

\draw[arrow] (v3012) -- (v3021);
\draw[arrow] (v3012) -- (v3102);

\draw[arrow] (v1320) -- (v2310);
\draw[arrow] (v1320) -- (v3120);

\draw[arrow] (v2130) -- (v2310);
\draw[arrow] (v2130) -- (v3120);

\draw[arrow] (v2301) -- (v2310);
\draw[arrow] (v2301) -- (v3201);

\draw[arrow] (v3021) -- (v3120);
\draw[arrow] (v3021) -- (v3201);

\draw[arrow] (v3102) -- (v3120);
\draw[arrow] (v3102) -- (v3201);

\draw[arrow] (v2310) -- (v3210);
\draw[arrow] (v3120) -- (v3210);
\draw[arrow] (v3201) -- (v3210);

\end{tikzpicture}

    \caption{The first ``unexpected" hypercube in $S_4$ of dimension $4$. This hypercube is also the first interesting example of our construction.}
    \label{fig: S4 hypercube}
\end{figure}

A similar phenomenon occurs in cluster algebras. It was long conjectured, and recently proved, that open Richardson varieties attached to pairs of permutations admit cluster structures \cite{CGGLSS, GLSBII}. A crucial invariant of a cluster variety is its number of frozen variables, which is roughly the maximal possible rank of a faithful torus action. 
Initially, it was thought that this number could be at most $n-1$ in $S_n$, as this is the dimension of the naturally acting maximal torus in $\mathrm{SL}(n)$. But one can find open Richardson varieties for $S_n$ that are tori of rank greater than $n-1$. The realization that this number of frozen variables is mysterious provided an important step in finding the cluster structure. After cluster structures on open Richardson varieties had been constructed, it was realised that the number of frozen variables is equal to the Kazhdan-Lusztig $d$-invariant (see \S~\ref{sec:R polynomials}), which is readily computable, and easily seen to occasionally exceed $n-1$.  The intervals of this paper provide examples where the number of frozen variables is $n \log n$, greatly exceeding any known family of examples.

A related phenomenon occurs in studying the combinatorial invariance conjecture of Dyer and Lusztig. Here the problem is to prove that isomorphic Bruhat intervals give rise to equal Kazhdan-Lusztig polynomials. Any Bruhat interval has a natural embedding inside $\mathfrak{h}^*$, the dual of the Lie algebra of a maximal torus. It is well known that one can recover the Kazhdan-Lusztig polynomial from this embedded graph. Thus, it is natural to try to prove the conjecture by relating embeddings associated to different incarnations of isomorphic intervals. Surprisingly, geometric embeddings of isomorphic intervals can have greatly differing dimensions. For example, hypercubes may appear embedded in  a space of dimension much lower than their ``natural" dimension. In an attempt to overcome this difficulty, Hone, Klein and the last author introduced a notion of a non-degenerate embedding, and study their moduli. In striking similarity to the cluster algebra story, the dimension in which this embedding lives is difficult to compute, appears related to the $d$-invariant, and may be considerably larger than $n-1$ (the dimension of $\mathfrak{h}^*$). The examples of this paper provide intervals where this ``natural dimension'' is much higher than naively expected.

A related motivation for the consideration of hypercubes in Bruhat order is again related to the combinatorial invariance conjecture. Motivated by AI experiments,  Blundell, Buesing, Davies, Veličković and the last author conjectured a new formula for Kazhdan-Lusztig polynomials based on the Bruhat graph \cite{davies2021advancing, BBDVW}, which (if true) would imply combinatorial invariance. This work has led to renewed interest in combinatorial invariance \cite{BrentiMarietti,GurevichWang,BGLower} and a proof for elementary intervals \cite{BG}. The formula relies on the choice of ``hypercube decomposition" of the interval, which provides an iterative way to build Bruhat intervals from hypercubes. In \cite{BBDVW}, a simple construction proves that one can do this in $n-1$ steps for any interval in $S_n$, however experimentally one can often be much more efficient. This efficiency is related to the existence of unexpectedly large families of hypercubes.

\begin{table}[ht]
\centering
\begin{tabular}{|r |r| r| r|}
\hline
$k$ & hypercubes & total & ratio \\
\hline \hline
3  & 223704 & 241620 & 0.9259 \\ \hline
4  & 286108 & 387932 & 0.7375 \\\hline
5  & 231484 & 498176 & 0.4647 \\\hline
6  & 111064 & 536860 & 0.2069 \\\hline
7  & 27484  & 502031 & 0.0547 \\\hline
8  & 2736   & 417142 & 0.0066 \\\hline
9  & 64     & 313063 & 0.0002 \\\hline
10 & 0      & 214478 & 0.0000 \\\hline
11 & 0      & 134933 & 0.0000 \\\hline
12 & 0      & 78104  & 0.0000 \\\hline
\hline
\end{tabular}
 \caption{ Hypercubes of length $k$ / total intervals of length $k$ for $S_7$}
    \label{fig: ratio Hyp}
\end{table}

\subsection{Methodology}

The large hypercube that is the subject of this paper was discovered using a novel experimental methodology.  Namely: we searched for pairs of permutations in $S_n$, not by any form of exhaustive search (which would be impractical for any but the smallest $n$) but using a protocol called AlphaEvolve, recently developed by Google DeepMind.  This protocol is an evolutionary algorithm in which programs evolve in time, using a large language model as the mode of reproduction, with ``fitness" determined by the program's ability to generate output that meets a user-specified mathematical objective.  In this case, we were searching for pairs of permutations whose Kazhdan-Lusztig $d$-invariant was as large as possible. 

We will describe our experimental process in some detail in Section~\ref{sec:experiments}, in the hope that this description will be useful for other mathematicians interested in adopting these methods. 

There is already quite a bit of evidence that AlphaEvolve and related protocols can be used to generate examples of mathematical interest~(see e.g. \cite{funsearch, generative, AE}.)  Quite naturally, a great deal of the research so far has focused on problems that have already been the subject of substantial mathematical effort; in dozens of cases, researchers have been able to use AlphaEvolve to match or incrementally improve upon the best examples in the existing literature. Typically (though not universally) the incremental improvements are difficult to interpret; for instance, the large cap set in $(\Z/3\Z)^8$ found in \cite{funsearch} is larger than any yet discovered, but so far we have not been able to concoct a satisfying story of ``why" this cap set is so large, and it is possible that no such story exists.

The present project is somewhat different. The problem of finding large $d$-invariants, or large hypercubes, is one of natural interest, but there is not a large existing literature about it.  And the example we found using AlphaEvolve is qualitatively different from the best one we knew, having a larger asymptotic order of growth in $n$.  What's more, the permutations (more precisely, the program that generated the permutations) is highly interpretable; within a few days we were able to explain to ourselves why it worked, and thus to prove that the machine-generated examples for small $n$ in fact generalized to all $n$.  In particular, our intuition is that the large hypercube presented in this paper is an object of mathematical interest in its own right, and that there is more to learn about and from it.

\section*{Acknowledgements}  
This paper involves essential contributions from Adam Zsolt Wagner, who cannot be named as a coauthor for technical reasons.
This material is based upon work supported by the National Science Foundation under Grant No. DMS-1929284 while JE, NL, DP, JS and GW were in residence at the Institute for Computational and Experimental Research in Mathematics in Providence, RI, during the ``Categorification and Computation in Algebraic Combinatorics'' semester program. 
JE's work was also supported by NSF grant DMS-2301386 and a Simons Fellowship. 
NL was partially supported by FONDECYT-ANID grant 1230247.  
DP was partially supported by FONDECYT-ANID grant 1240199.
JS was partially supported by UNAM's PAPIIT Grant IA102124 and SECIHTI Project C-F-2023-G-106.
GW was supported by ARC project DP230102982 and the Max Planck-Humboldt research award.
The authors would also like to thank AlphaEvolve for stimulating discussions. We are grateful to Leonardo Patimo for pointing us to relevant references, and to Mario Marietti for some helpful email exchanges.

\section{Dyadically well-distributed permutations } \label{sec:dwd}

In this section we introduce the main object of study of the paper: the class of dyadically well-distributed permutations.
We establish their fundamental properties and determine the cardinality of the set of such permutations. 
In particular, we show that this class forms a Bruhat interval isomorphic to a hypercube, precisely the large hypercube mentioned in the introduction.

\subsection{Basic properties}
Throughout the paper we fix a positive integer $n$. 
We denote by $[n]$ the interval with $n$ elements $\{0, \dots, n-1\}$.
We think of $S_n$ as the set of bijections on $[n]$. 
Henceforth we will assume that $n=2^m$.

\begin{definition}\label{def: dwd}
    For $k\geq 0$, a \emph{basic $k$-interval} is an interval in $[2^m]$ of the form $[c 2^k, (c+1) 2^k - 1]$. 
    Henceforth, we identify $\mathbb{F}_2^m$ with $[2^{m}]$  via 
    \begin{equation}
        (a_0,a_1,\ldots , a_{m-1}) \mapsto \sum_{i=0}^{m-1}a_{i}2^{m-1-i}. 
    \end{equation}
Under this identification, a basic $k$-interval is determined by fixing the first $m-k$ bits.

   For a basic $k$-interval $S$, we let $P_S$ denote the unique binary string of length $m-k$ that appears as the first $m-k$ bits of every element of $S$.

We say that a permutation $\pi$ is \emph{dyadically well-distributed} if, for every basic $k_1$-interval $S$ and $k_2$-interval $T$, with $k_1 + k_2 = m$, there exists exactly one element $i \in S$ such that $\pi(i) \in T$. 

We denote by $\dwd{m}$ the set of all dyadically well-distributed permutations in $S_{2^m}$. 
\end{definition}

\begin{remark}\label{rmk: pictorial-dwd}
For  $\pi \in S_{2^m}$ we define its permutation matrix  $M_{\pi}=(m_{i,j})$ by $m_{i,j}=1$ if $\pi (j)=i$ and $0$ elsewhere. 
Identifying an element $\pi \in S_{2^m}$ with its permutation matrix $M_{\pi} \in \mathrm{Mat}_{2^m \times 2^m}$, we obtain a more pictorial interpretation of dyadically well-distributed permutations. For each $k = 0,\dots,m$, tile $M_{\pi}$ into blocks of size $2^{k} \times 2^{m-k}$, starting from the upper-left corner and proceeding in a left-to-right, top-to-bottom fashion (i.e., in reading order).  Then $\pi\in \dwd{m}$ if and only if each such block contains exactly one entry equal to $1$ (see Figure \ref{fig: x_m and y_m for m=4}).

When $k=0$ and $k=m$ the above condition is automatically satisfied for any permutation. Thus, to check whether $\pi \in \dwd{m}$ it is enough to check the above condition for $0<k<m$. 
 \end{remark}

 \begin{definition}
     By a \emph{fundamental block} in a matrix $M \in \mathrm{Mat}_{2^m \times 2^m}$ we will mean a $(2^k \times 2^{m-k})$-block as considered in Remark \ref{rmk: pictorial-dwd}.
 \end{definition}

\begin{definition}\label{def:x-y}
We inductively define elements $x_m\in S_{2^m}$ as follows:
\begin{align*}
x_1 & = ({\color{blue}0}, {\color{red}1}) \in S_2, \\
x_2 & =({\color{blue}0}, {\color{red}2}, {\color{blue}1}, {\color{red}3}) \in S_4, \\
x_3 & = ({\color{blue}0}{, \color{red}4}, {\color{blue}2}, {\color{red}6}, {\color{blue}1}, {\color{red}5}, {\color{blue}3}, {\color{red}7})\in S_8, \\
 & \vdots
\end{align*}
In other words, $$x_{m+1}=({\color{blue}x_m(0)},{\color{red}x_m(0)+2^m}, {\color{blue}x_m(1)},{\color{red}x_m(1)+2^m},\ldots,
{\color{blue}x_m(2^m-1)},
{\color{red} x_m(2^m-1) + 2^m}) \in S_{2^{m+1}}.$$
Define $y_m$ as the reverse element in one line-notation, i.e.
\[
y_m=(x_m(2^m-1),\cdots, x_m(1),x_m(0)).\]
\end{definition}

\begin{example}
\label{exm4}
    Let $m=4$. In this case we have
    \begin{equation}
    \begin{array}{l}
         x_4=(0, 8, 4, 12, 2, 10, 6, 14, 1, 9, 5, 13, 3, 11, 7, 15)   \\
         y_4=(15, 7, 11, 3, 13, 5, 9, 1, 14, 6, 10, 2, 12, 4, 8, 0) .
    \end{array}
    \end{equation}
    The permutation matrices of $x_4$ and $y_4$ are depicted in Figure \ref{fig: x_m and y_m for m=4}.
    This figure also shows the fundamental blocks, making it evident that $x_{4}$ and $y_{4}$ are elements of $\dwd{4}$.
\end{example}

\begin{figure}[htbp]
\centering
\begin{subfigure}{0.33\textwidth}
\centering

\DrawGridAndPoints{8}{2}{Fundamental bocks of size $2^3\times 2^1$}
\end{subfigure}\hfill
\begin{subfigure}{0.33\textwidth}
\centering

\DrawGridAndPoints{4}{4}{Fundamental bocks of size $2^2\times 2^2$}
\end{subfigure}\hfill
\begin{subfigure}{0.33\textwidth}
\centering

\DrawGridAndPoints{2}{8}{Fundamental blocks of size $2^1\times 2^3$}
\end{subfigure}

    \caption{$M_{x_4}$ (blue) and $M_{y_4}$ (red) with fundamental blocks highlighted.}
    \label{fig: x_m and y_m for m=4}
\end{figure}

\begin{remark} \label{rem: x and y as permutation matrices}
Let us explain why the permutations $x_m$ and $y_m$ are more natural than they may first appear.
Identifying the set $[2^m]$ with $\mathbb{F}_2^m$ as before,  the permutations  $x_m$ and $y_m$ correspond to affine maps in $\mathrm{GL}(\mathbb{F}_2^m) \ltimes \mathbb{F}_2^m\subset S_n $ defined by the formulae:
\[
x_m(a_0, \dots, a_{m-1}) = (a_{m-1}, \dots, a_0) \qquad \mbox{and} \qquad
y_m(a_0, \dots, a_{m-1}) = (\overline{a_{m-1}}, \dots, \overline{a_0}), 
\]
where $\overline{a_i}=a_i+1$. 
\end{remark}

For $a,b\in [n]$ we denote by $(a,b)$ the transposition that interchanges $a$ and $b$. 

\begin{definition}
    Let $\pi =(x_1,\ldots , x_n)\in S_n$.
    We define the \emph{length} of $\pi$, $\ell(\pi)$,  as the number of inversions of $\pi$. 
    That is,
    \begin{equation}
        \ell (\pi) =\; \# \{ (i,j) \in [n]^2 \mid i<j \mbox{ and }  \pi (i)> \pi (j) \} .
    \end{equation}
\end{definition}

In the following lemma we collect some basic facts about permutations $x_m$ and $y_m$. 
\begin{lemma}\label{xdwd}
Let $m\geq 1$. 
Then we have
\begin{enumerate}
    \item \label{unouno} $x_m$ and $y_m$ are involutions.
    \item \label{dosdos} $x_my_m=y_mx_m=w_0$, where $w_0$ denotes the longest element in $S_{2^{m}}$.
    \item $x_m, y_m \in \dwd{m}$.
    \item $\ell(x_m) =  2^{m-2} (2^m-(m+1))  $ and $\ell (y_m)= 2^{m-2}(2^m +(m-1)) $. Thus, $\ell (y_m) -\ell (x_m)=m2^{m-1}$. 
\end{enumerate}
\end{lemma}

\begin{proof}
The first two claims follow directly from Remark \ref{rem: x and y as permutation matrices}.

We now focus on the third claim. 
By Items \ref{unouno} and \ref{dosdos} we have $y_{m}=x_mw_0$. 
Therefore, it is enough to show $x_m\in \dwd{m}$. 

We fix   $k \in \{ 0,\dots,m\}$. 
Let  $S$ and $T$ be basic $(m-k)$- and $k$-intervals, respectively. 
Let $P_S \in \F_2^{k}$ and $P_T \in \F_2^{m-k}$ be the binary sequence associated to $S$ and $T$, respectively. 
In order to show that $x_m$  belongs to $\dwd{m}$ we must verify that there is a unique  $i \in S$  such that $x_m(i) \in T$. 
It is easy to see that the unique $i$ that satisfies both conditions is the one associated to the $m$-bit binary sequence $P_SP_T^{-1}$, where
$P_{T}^{-1}$ denotes the reverse of $P_T$. 
We have thus proved $x_m\in \dwd{m}$. 

It remains to prove the length formula for $x_m$; the case of $y_m$ being analogous.
We recall that $\ell(x_m)$ coincides with the number of inversions, i.e. the number of pairs $(i,j)$ with $i<j$ such that $x_m(i)>x_m(j)$.

It is immediate that $\ell(x_1)=0$ and $\ell(x_2)=1$.
For $m\ge 3$, we claim that
\begin{equation} \label{eq: recursion for length x_m}
   \ell(x_m)= 2 \ell(x_{m-1})+ \bigl(2^{m-1}-1\bigr)2^{m-2}. 
\end{equation}

To see this, write the one-line notation of $x_m$ as a concatenation
$x_m = (a_m \mid b_m)$,
where each block has length $2^{m-1}$.
By construction, both $a_m$ and $b_m$ encode copies of $x_{m-1}$, so together they contribute $2 \ell(x_{m-1})$ inversions.

Let $C_m$ be the number of inversions of $x_m$ involving an entry in $a_m$ and one in $b_m$.
Since entries in $a_m$ are even and those in $b_m$ are odd, each entry $k\in a_m$ contributes exactly $k/2$ to $C_m$.
Summing over all even numbers from $0$ to $2^{m}-1$, we obtain
\begin{equation}
C_m = (2^{m-1}-1)\cdot 2^{m-2}.
\end{equation}
This proves the recursion \eqref{eq: recursion for length x_m}.

Solving the recurrence with the initial conditions $\ell(x_1)=0$, $\ell(x_2)=1$, we find
\begin{equation}
\ell(x_m)=2^{m-2}\bigl(2^m-(m+1)\bigr),
\end{equation}
as desired.
\end{proof}

 The Bruhat order in the symmetric group is defined as follows. 
\begin{definition} \label{def:Bruhat}
    Let $\pi_1 , \pi_2 \in S_n$. We write $\pi_1 \leq \pi_2$ if there exists a sequence of transpositions $t_1, \ldots , t_k\in S_n$ such that 
    \begin{enumerate}
        \item $\pi_1 t_1\cdots t_k = \pi_2$;
        \item $\ell (\pi t_1\cdots t_{s-1}) <\ell (\pi_1 t_1\cdots t_{s})$ for all $1\leq s \leq k$. 
    \end{enumerate}
\end{definition}

\begin{definition} \label{def: Ehresmann matrix}
Let $w \in S_n$. 
We define the Ehresmann matrix $E_w$ associated to $w$ as 
\begin{equation}
  E_w(i,j) \;=\; \#\{  k \leq j \mid w(k) \geq i  \},
\qquad 0 \leq i,j < n.  
\end{equation}
That is, $E_w(i,j)$ counts how many values among $0, 1, \dots, j$ are taken by $w$ to values greater than or equal to $i$.
Equivalently, it counts the $1$'s
of the permutation matrix of $w$ lying in the southwest rectangle determined by $(i,j)$.    
\end{definition}

The following criterion, originally due to Ehresmann, is a useful tool for comparing two permutations in the Bruhat order. 
A proof (using slightly different conventions) can be found in \cite[Theorem 2.1.5]{bjorner2005combinatorics}. 

\begin{lemma} \label{lem: Ehresmann}  
Let $x, y\in S_n$. 
Then, we have
\begin{equation}
    x \leq y \text{ in the Bruhat order}
\quad\Longleftrightarrow\quad
E_x(i,j) \le E_y(i,j)
\quad\text{for all $i,j$.}
\end{equation}
\end{lemma}

For $x,y\in S_n$ we denote by $[x,y]$ the corresponding Bruhat interval  
$$[x,y]= \{ z\in S_n \mid x\leq z\leq y  \}.$$

The relation between dyadically well-distributed permutations and $x_m$ and $y_m$ is the following: 
\begin{equation}
   \dwd{m} =[x_m,y_m].
\end{equation}

For the moment we are in position to prove only one inclusion. The proof of the other inclusion is postponed to the next section. 

 \begin{lemma} \label{lem: one-inclusion}
     For all $m\geq 1$ we have $[x_m,y_m]\subset   \dwd{m} $. 
 \end{lemma}

\begin{proof}
Let  $\pi \in [x_m,y_m]$ and $B$ be a fundamental block. 
Let  $U_\pi(B)$ denote  the number of $1$'s inside the block $B$ in the permutation matrix $M_{\pi}$. 
By Remark~\ref{rmk: pictorial-dwd} ,   in order to prove that $\pi\in \dwd{m}$,  it is enough to show that $U_{\pi} (B)=1$. 

Let $V_1, V_2, V_3$ and $V_4$ be the vertices of $B$ labeled in reading order, as is shown in Figure \ref{fig:divide-B}.
The inclusion-exclusion principle yields  
\begin{equation}\label{eq: inc-exc containing}
    U_\pi(B) = E_{\pi}(V_2) -E_{\pi}(V_1) - E_{\pi}(V_4)+E_{\pi} (V_3). 
\end{equation}
By Lemma \ref{xdwd} we know that $x_m, y_m \in \dwd{m}$. 
In particular, via Remark \ref{rmk: pictorial-dwd}, we obtain $E_{x_m}(V_{i}) = E_{y_m}(V_i)$ for all $1\leq i \leq 4$. 
On the other hand,  Lemma \ref{lem: Ehresmann}  yields 
\begin{equation}
    E_{x_m}(V_i) \leq E_{\pi} (V_i) \leq E_{y_m}(V_i),
\end{equation}
  for all $1\leq i \leq 4$. We conclude that $E_{x_m}(V_i) =E_\pi(V_i)$.  
  By combining these equalities with \eqref{eq: inc-exc containing} we obtain $U_{\pi}(B)=U_{x_m}(B)=1$, as we  wanted to show. 
\end{proof}

\begin{figure}
    \centering
    \begin{tikzpicture}[scale=1]
  \coordinate (V1) at (0,0);
  \coordinate (V2) at (2,0);
  \coordinate (V3) at (0,-2);
  \coordinate (V4)  at (2,-2);

  \draw (V1) -- (V2) -- (V4) -- (V3) -- cycle;


  \fill (V1) circle (2pt) node[above left]  {$V_1$};
  \fill (V2) circle (2pt) node[above right]       {$V_2$};
  \fill (V3) circle (2pt) node[left]        {$V_3$};
  \fill (V4)  circle (2pt) node[above right] {$V_4$};

  \node at (1,-1) {$B$};
\end{tikzpicture}
    \caption{ Vertices in a fundamental block }
    \label{fig:divide-B}
\end{figure}
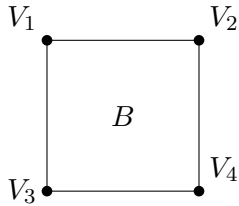

\subsection{A big Bruhat hypercube}

Let us introduce a special type of tiles in a $2^m \times 2^m$-matrix. 

\begin{definition}
We say that a basic $k_1$-interval $S$ and a basic $k_2$-interval $T$ are  \emph{complementary} if $k_1 + k_2 = m+1$ and $k_1, k_2 \geq 1$.
A \emph{complementary block} is a block in a $2^m \times 2^m$-matrix whose entries are indexed by $S \times T$, where $S$ and $T$ are complementary intervals. We denote by $\comI{m}$ the set of complementary blocks. 
\end{definition}

\begin{definition}\label{def: order-iso}
    Let $S_2=\{e,s\}$. 
    We define a function
    \[
    \varphi: \dwd{m} \to S_2^{\comI{m}}
    \]
    as follows. Let $S$ be a basic $k_1$-interval and $T$ a  basic $k_2$-interval such that $(S,T)\in \comI{m}$.  The submatrix $M_{(S,T)}$ of $M_\pi$ corresponding to $(S,T)$ is the union  of two disjoint fundamental blocks of size $2^{k_1-1} \times 2^{k_2 }$.
  Thus, if $\pi \in \dwd{m}$, there are exactly two entries equal to $1$ in $M_{(S,T)}$. After deleting all zero columns and rows in $M_{(S,T)}$, we end up in one of the following cases:
\[
\begin{pmatrix} 1 & 0 \\ 0 & 1 \end{pmatrix}  \qquad \text{or} \qquad \begin{pmatrix} 0 & 1 \\ 1 & 0 \end{pmatrix}.
\]
In the first case, we set $\varphi(\pi)(S,T) = e$, and in the second case we set $\varphi(\pi)(S,T) = s$.
\end{definition}

We can see $S_2^{\comI{m}}$ as partially ordered set as follows. 

\begin{definition}
    Let $m\geq 1$ and  $f_1,f_2\in S_2^{\comI{m}}$. 
    We write $f_1\leq f_2$ if 
    \begin{equation}
        f_1(P,Q) \leq f_2(P,Q)
    \end{equation}
    for all $(P,Q) \in \comI{m}$. 
    In other words,  $f_1\leq f_2$ if and only if for all $(P,Q) \in \comI{m}$ we have 
    \begin{equation}
        f_1(P,Q) = s \implies f_2(P,Q) =s. 
    \end{equation}
\end{definition}

We notice that $(S_2^{\comI{m}}, \leq )$ is isomorphic as a poset to a hypercube of rank $|\comI{m}|=m2^{m-1}$. 
The minimal and maximal element in $(S_2^{\comI{m}}, \leq )$, which we denote by $f_e$ and $f_s$, satisfy $f_e(P,Q)=e$ and $f_s(P,Q)=s$ for all $(P,Q)\in \comI{m}$. 

\begin{lemma}\label{lem:order-preserving}
The function $\varphi$ is order-preserving.
This is, for all $\pi_1,\pi_2 \in \dwd{m}$ such that $\pi_1 \leq \pi_2$ in the Bruhat order we have $\varphi (\pi_1) \leq \varphi (\pi_2)$ in $(S_2^{\comI{m}}, \leq)$. 
\end{lemma}

\begin{proof}
    We argue by contraposition. 
    Suppose that $\varphi(\pi_1) \not\leq \varphi(\pi_2)$. 
    Then there exists a complementary pair $(S,T)\in \comI{m}$ such that
    \begin{equation}\label{eq:no-menor}
        \varphi(\pi_1)(S,T)=s 
        \qquad\text{and}\qquad 
        \varphi(\pi_2)(S,T)=e. 
    \end{equation}
    Divide the complementary pair $(S,T)$ into four congruent rectangles, as illustrated in Figure \ref{fig:divide-CB}. 
    Let $R$ be the bottom-left rectangle in this subdivision, and let $V_1$, $P$, $V_2$, and $V_3$ be the vertices of $R$ listed in reading order. 
    Note that the points $V_i$ are vertices of fundamental blocks. 
    Since $\pi_1,\pi_2\in \dwd{m}$, we have $E_{\pi_1}(V_i)=E_{\pi_2}(V_i)$ for all $i\in\{1,2,3\}$. 
    
    Let $X_{\pi_i}(R)$ denote the number of $1$'s inside $R$ in the permutation matrix $M_{\pi_i}$. 
    By \eqref{eq:no-menor}, we have $X_{\pi_1}(R)=1$ and $X_{\pi_2}(R)=0$. 
    By inclusion--exclusion, we have
    \begin{equation}
        E_{\pi_i}(P)
        = E_{\pi_i}(V_1) + E_{\pi_i}(V_3) - E_{\pi_i}(V_2) + X_{\pi_i}(R).
    \end{equation}
    It follows that $E_{\pi_1}(P)-E_{\pi_2}(P)  =1 $. 
    In particular,  $E_{\pi_2}(P) < E_{\pi_1}(P)$.
    Therefore, Lemma~\ref{lem: Ehresmann} yields 
    $\pi_1 \not\leq \pi_2$, as we wanted to show.
\end{proof}

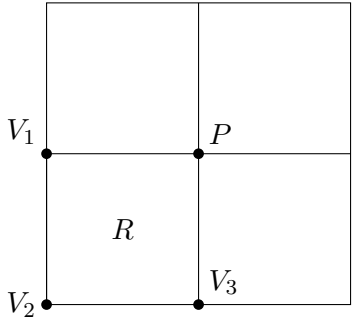
\begin{figure}
    \centering
    \begin{tikzpicture}[scale=1]
  \coordinate (V1) at (0,0);
  \coordinate (P) at (2,0);
  \coordinate (V3) at (0,-2);
  \coordinate (V4)  at (2,-2);
  \coordinate (B1) at (0,-2);
  \coordinate (B2) at (2,2);
  \coordinate (B3) at (4,-2);
  \coordinate (T1) at (0,2);
  \coordinate (T2) at (2,2);
  \coordinate (T3) at (4,2);
  \coordinate (XX) at (4,0);

  \draw (T1) -- (T3) -- (B3) -- (B1) -- cycle;

  \draw (V1) -- (V3);   
  \draw (V4) -- (B2);   
  \draw (V1) -- (XX);   

  \fill (V1) circle (2pt) node[above left]  {$V_1$};
  \fill (P) circle (2pt) node[above right]       {$P$};
  \fill (V3) circle (2pt) node[left]        {$V_2$};
  \fill (V4)  circle (2pt) node[above right] {$V_3$};

  \node at (1,-1) {$R$};
\end{tikzpicture}
    \caption{Dividing a complementary block as in the proof of Lemma \ref{lem:order-preserving}}
    \label{fig:divide-CB}
\end{figure}

\begin{lemma} \label{lem:min-max}
    For all $m\geq 1$ we have
    $\varphi (x_m) = f_e$ and $\varphi (y_m) = f_s$. 
\end{lemma}

\begin{proof}
Let us show the statement for $x_m$, the one for $y_m$ is analogous. 
Let $(S,T)$ be a complementary pair. 
We need to show that $\varphi(x_m){(S,T)} = e$.
Note that this means that, if $a < b \in S$ are the two elements such that $x_m(a), x_m(b) \in T$, then $x_m(a) < x_m(b)$.

We use the interpretation of $x_m$ as an endomorphism of $\mathbb{F}_2^{m}$.
Assume $S$ is a basic $k$-interval, and $T$ is a basic $m-k+1$-interval, with $k \geq 1$. 
Let $P_S\in \F_2^{m-k}$ and $P_T\in \F_2^{k-1}$ be the binary sequences associated to $S$ and $T$, respectively. 

If $k = 1$, then $S=\{P_S0,P_S1\}$ and $T = \mathbb{F}_2^{m}$. Then, $a=P_S 0$ and $b=P_S 1$ (since $a<b$). 
Thus, $x_m(a)=0 P_S^{-1}$ and $x_m(b)=1 P_S^{-1}$.
It follows that $x_m(a)=P_S0 < P_S1=x_m(b)$, since the natural order on the interval $[2^m]$ becomes the lexicographic order on $\F_2^m$, via our identification $[2^{m}] \leftrightarrow \F_2^m$.

We now assume $k > 1$.
In this case the two elements in $S$ mapping to $T$ are 
\[a = P_S0P_T^{-1}  < P_S1P_T^{-1} =b .\]
It is clear that $x_m(a)= P_T0P_S^{-1} <P_T1P_S^{-1}= x_m(b)$ and this finishes the proof.  
\end{proof}

\begin{definition}
  Let $m\geq 1$ and $1\leq k \leq m$. 
  For any $k$-bit binary string $B$ we write $B\star$ for a $m$-bit  binary string  starting with $B$ (so that $\star$ is a binary string of length $m-k$).

 Let $\pi \in \dwd{m}$  and $(S,T) \in \comI{m}$.
 Let $P_S$ and $P_T$ be the binary sequences associated to $S$ and $T$, respectively. 
 Let  $P_S\star_1,P_S\star_2 \in \F_2^m$ be the two $m$-bit binary strings starting with $P_S$ which are sent by $\pi$ to strings starting with $P_T$.
 We define the {\em flip} $Fl_{S,T}(\pi)$ to be the permutation which agrees with $\pi$ on all strings other than $P_S\star_1$ and $P_S\star_2$, and which sends $P_S \star_1$ to $\pi(P_S\star_2)$ and $P_S\star_2$ to $\pi (P_S\star_1)$.  In other words, $Fl_{S,T}(\pi)$ is $\pi$ multiplied on the right by the transposition $(P_S\star_1,P_S\star_2)$.
\end{definition}

In the following lemma we identify $S_2$ with $\Z/2\Z$, sending $e$ to $0$ and $s$ to $1$.
We can thus identify $S_2^{\comI{m}}$ with the set of functions $f: \comI{m} \rightarrow \Z/2\Z$.

    \begin{lemma}\label{lem: Jordan}
    Let $\pi \in \dwd{m}$ and $(S,T) \in \comI{m}$.
    Let $\chi_{S,T}: \comI{m}\rightarrow \Z/2\Z$ be the  function sending $(S,T)$ to $1$ and every other complementary block to $0$.
    Then, we have
    \begin{enumerate}
        \item  \label{UNO} $Fl_{S,T}(\pi) \in \dwd{m}$.
        \item  \label{DOS} $\varphi(Fl_{S,T}(\pi)) = \varphi(\pi) + \chi_{S,T}$.
        \item  \label{TRES}\begin{equation}\label{eq:flipinversion}
    \ell(Fl_{S,T}(\pi)) = \left\{\begin{array}{rl}
        \ell(\pi) +1,    & \mbox{if }\varphi(\pi)(S,T) = 0;  \\[3pt]
        \ell(\pi) -1,    & \mbox{if }\varphi(\pi)(S,T) = 1.
    \end{array}
    \right.
    \end{equation}
    \end{enumerate}
    \end{lemma}

 \begin{proof}
Let $\sigma \coloneqq Fl_{S,T}(\pi)  \in \dwd{m}$.
By definition of the flip, the permutation matrices $M_\pi$ and $M_{\sigma}$ are identical outside the complementary pair $(S,T)$.
        Within that block, they differ by a local move
        \begin{equation}\label{eq: covering-switch}
\begin{pmatrix} 1 & \cdots & 0 \\
\vdots &  \mbox{ \Large 0}  & \vdots \\ 0 & \cdots & 1 \end{pmatrix} \longleftrightarrow  \begin{pmatrix} 0 & \cdots & 1 \\
\vdots &  \mbox{ \Large 0}  & \vdots \\ 1 & \cdots & 0 \end{pmatrix}.
\end{equation}
We prove each item separately.
\begin{enumerate}
\item Following our established notation, we write $\pi(P_S \star_1) = P_T *_1$ and $\pi(P_S \star_2) = P_T *_2$ for the two $1$ entries in the permutation matrix of $\pi$ which are moved by the flip.

Suppose that  $\sigma \notin \dwd{m}$.  Then there is a fundamental block $(U,V)$ such that the equation $\sigma(P_U \star) = P_V \star'$ has no solution.  Since $\pi \in \dwd{m}$, the equation $\pi(P_U \star) = P_V \star'$ {\em does} have a solution, and indeed a unique solution.  Since $\sigma$ differs from $\pi$ only on the inputs $P_S \star_1$ and $P_S \star_2$, this solution must be either $\pi(P_S \star_1) = P_T *_1$ or $\pi(P_S \star_2) = P_T *_2$.  Without loss of generality, we assume it is the former.  Then we know $P_S \star_1$ starts with $P_U$ and $P_T *_1$ starts with $P_V$.

By the definition of complementary pair and fundamental block, we know that $|P_S| + |P_T| = m-1$ and $|P_U| + |P_V| = m$.  This implies that either $|P_U| \leq |P_S|$ or $|P_V| \leq |P_T|$.  In the former case, the fact that $P_S \star_1$ starts with $P_U$ now implies that $P_U$ is a prefix of $P_S$.  In particular, $P_S \star_2$ starts with $P_U$.  By the definition of the flip,
$$
\sigma(P_S \star_2) = P_T *_1
$$
and since $P_T *_1$ starts with $P_V$, this provides a solution to $\sigma(P_U \star) = P_V \star'$, contradicting our hypothesis. If on the other hand $|P_V| \leq |P_T|$, then $P_V$ is a prefix of $P_T$ and the equation $\sigma(P_S \star_1) = P_T *_2$ provides a contradiction in exactly parallel fashion.  We conclude that $\sigma = Fl_{S,T}(\pi)$ lies in $\dwd{m}$, as claimed.

\item 
  We now compute  $\varphi(\sigma)$.
  We first note that 
$$\varphi(\sigma)(S,T)= \varphi(\pi)(S,T) + 1,$$
because the flip switches the positions of the two $1$'s in the block $(S,T)$. 

  It remains to show is that the flip does not change the value of $\varphi(\pi)(S_1,T_1)$ for any other complementary pair $(S_1,T_1)\neq (S,T)$.
  The key point is that, by the definition of $\varphi$, the value of $\varphi(\sigma)(S_1,T_1)$ is determined by the knowledge of {\em either one} of the two $1$'s in the complementary pair $(S_1,T_1)$. 
  For instance, if one of these $1$'s is located in the upper left quarter of $(S_1,T_1)$ then the other $1$ in $(S_1,T_1)$ is forced to be located in the lower right quarter, and $\varphi(\sigma)(S_1,T_1) = 0$. 
  
  Now of the two $1$'s lying in $(S_1,T_1)$, at most one of them can lie in $(S,T)$, since $(S,T) \cap (S_1,T_1)$ is contained in a fundamental block and can thus contain at most one $1$. 
  This means that there is some $1$ such that it lies in $(S_1,T_1)$ but does {\em not} lie in $(S,T)$. 
  Since $M_{\pi}$ and $M_{\sigma}$ coincide outside $(S,T)$ we conclude that this $1$ has the same location in both $M_{\pi}$ and $M_{\sigma}$.
   Since the position of this common $1$ determines the value of $\varphi(\sigma)(S_1,T_1)$ and $\varphi(\pi)(S_1,T_1)$, we conclude that the two values are equal, as claimed.

\item   Finally, we compute $\ell (\sigma)$, this is, the number of inversions of $\sigma$.  
Since $M_\pi$ and $M_{\sigma}$ differ only by a local move as in \eqref{eq: covering-switch}, it is clear that $\ell(\sigma) = \ell(\pi) \pm 1$. 

Let us be more precise. 
Let $i,j \in [2^m]$ be the columns where the two $1$'s are located. 
Suppose that in \eqref{eq: covering-switch} the left (resp. right) matrix corresponds to a sub-matrix of $M_{\pi}$ (resp. $M_{\sigma}$). 
In this case we have $\varphi (\pi)(S,T)=0$.
On the other hand,  the pair $(i,j)$ is an inversion of $\sigma $ but not an inversion of $\pi$.
Therefore, $\ell (\sigma) = \ell (\pi) +1$. 
This proves the first case of \eqref{eq:flipinversion}.
The other case is analogous.      
\end{enumerate}
Having proved the three claims the proof is complete. 
\end{proof}

\begin{prop}\label{prop: phi is iso}
The map $\varphi: \dwd{m} \to S_2^{\comI{m}}$ is a poset isomorphism, i.e., $\pi_1 \leq \pi_2$ if and only if $\varphi(\pi_1) \leq \varphi(\pi_2)$. 
\end{prop}

\begin{proof}
We split the proof in three steps. 
\begin{enumerate}
    \item $\varphi$ is injective. 

Let $\pi \in \dwd{m}$ and $a \in [2^{m}]$.
We show that $\pi(a)$ can be recovered uniquely from $\varphi(\pi)$.

Let $C_1$ denote the column indexed by $a$.
Consider the unique complementary $(m,1)$-block $(S_1,T_1)$ containing $C_1$.
The value $\varphi(\pi)(S_1,T_1)$ specifies which of the two sub-columns of height $2^{m-1}$ cannot contain $\pi(a)$.
Thus, this step \emph{blocks} exactly $2^{m-1}$ possible rows in $C_1$.
Let $C_2$ be the remaining (unblocked) sub-column, which has height $2^{m-1}$.

Next, consider the unique complementary $(m-1,2)$-block $(S_2,T_2)$ containing $C_2$.
The value $\varphi(\pi)(S_2,T_2)$ again blocks half of $C_2$, that is, a sub-column of $2^{m-2}$ rows.
Let $C_3$ be the remaining unblocked sub-column, now of height $2^{m-2}$.

Continuing in this way, at the $k$-th step we block $2^{m-k}$ rows.
After $m$ steps, the total number of blocked rows is
\begin{equation}
  2^{m-1} + 2^{m-2} + \cdots + 2^{1} + 2^{0} \;=\; 2^{m} - 1.
\end{equation}
Since the entire column $C_1$ contains $2^{m}$ rows, blocking $2^{m} - 1$ of them leaves exactly one row unblocked, which must be the position of $\pi(a)$.

Therefore, $\pi(a)$ is uniquely determined by $\varphi(\pi)$ for every $a \in [2^{m}]$, and hence $\varphi$ is injective.

\item $\varphi$ is surjective. 

We use the identification $S_2\leftrightarrow\Z/2\Z$ as in the proof of the previous lemma. 

Let $f: \comI{m}\rightarrow \Z/2\Z$ and  
\begin{equation}
    \mathcal{A}_f = \{ (S,T) \in \comI{m} \mid f(S,T)=1   \}.
\end{equation}
We prove that any $f$ has a pre-image by induction on $|\mathcal{A}_f|$. 

If $|\mathcal{A}_f| = 0$, then $f = f_e = f_0$. 
Thus Lemma~\ref{lem:min-max} gives $\varphi(x_m) = f$. 
This establishes the base case of the induction.

Now fix $f$ such that $|\mathcal{A}_f| > 0$. 
Let $(S,T) \in \mathcal{A}_f$ and set $\mathcal{A}' = \mathcal{A}_f \setminus \{(S,T)\}$. 
Define $g : \comI{m} \to \mathbb{Z}/2\mathbb{Z}$ to be the function that takes the value $1$ on $\mathcal{A}'$ and $0$ elsewhere. 
By the induction hypothesis, there exists $\pi_g \in \dwd{m}$ such that $\varphi(\pi_g) = g$. 
It follows from Lemma~\ref{lem: Jordan} that 
\[
\varphi\bigl(Fl_{(S,T)}(\pi_g)\bigr) = f.
\]
This completes the induction.

\item $\varphi$ is an isomorphism of posets. 

Having already proved that $\varphi$ is bijective, it remains to show that $\varphi$ preserves the order in both directions.
One direction is established in Lemma \ref{lem:order-preserving}.
Thus, it suffices to prove that
\[
\varphi(\pi) \leq \varphi(\pi') \;\implies\; \pi \leq \pi',
\qquad \text{for } \pi, \pi' \in \dwd{m}.
\]
The key observation is that $\varphi(\pi) \leq \varphi(\pi')$ implies
$\mathcal{A}_{\varphi(\pi)} \subset \mathcal{A}_{\varphi(\pi')}$.
A straightforward inductive argument on 
$\bigl|\mathcal{A}_{\varphi(\pi')} \setminus \mathcal{A}_{\varphi(\pi)}\bigr|$,
together with Items \ref{DOS} and \ref{TRES} in Lemma\ref{lem: Jordan},
shows that $\pi'$ can be obtained from $\pi$ through a sequence of
length-increasing right multiplications by transpositions.
Therefore, $\pi \leq \pi'$, as desired.
\end{enumerate}
\end{proof}

\begin{theorem} \label{thm: main}
For all $m\geq 1$ we have: 
\begin{enumerate}
    \item  $\dwd{m} = [x_m,y_m]$.
    \item  $[x_m,y_m] \simeq S_2^{\comI{m}}$. 
    \item The Bruhat interval $[x_m,y_m]$ is a hypercube of  rank $m2^{m-1}$. 
\end{enumerate}   
\end{theorem}

\begin{proof}
 By Lemma \ref{lem: one-inclusion} to show the first claim  it is enough to prove the inclusion $\dwd{m} \subset [x_m,y_m]$.  
 Let $\pi \in \dwd{m}$. We have $f_e\leq \varphi (\pi) \leq f_s$. 
 By Lemma \ref{lem:min-max} and Proposition \ref{prop: phi is iso} we obtain $x_m \leq \pi \leq y_m$. This is $\pi \in [x_m,y_m]$. 
 This finishes the proof of the first claim. 
The second claim follows by combining Proposition \ref{prop: phi is iso} and the first claim. 
Finally, the third claim is a direct consequence of the second once we recall that $|\comI{m}|=m2^{m-1}$. 
\end{proof}

\section{Connections, extensions and applications}

In this section we point out some connections with the existing literature, explain a $p$-adic generalization, and give some applications to cluster algebras and moduli of embeddings of Bruhat graphs.

\subsection{Low-discrepancy sequences}
\label{ss:lowdisc}

The sequence
\begin{equation*}
(0, 8, 4, 12, 2, 10, 6, 14, 1, 9, 5, 13, 3, 11, 7, 15)
\end{equation*}
appearing in Example~\ref{exm4} is known (after dividing by 16 and removing the initial $0$) as the {\em van der Corput sequence} with denominator $16$ in $[0,1]$.  It is a well-known example of a low-discrepancy or quasirandom sequence in the unit interval.

The dyadically well-distributed permutations also have an incarnation in the world of low-discrepancy sequences.  If $\pi$ is a permutation in $\dwd{m}$, we may think of its permutation matrix as a subset of $2^m$ points $(i/2^m,\pi(i)/2^m)$ in the unit square.  A size-$2^m$ subset of the square, like this one, such that each fundamental block contains exactly one of the points is called a {\em $(0,m,2)$-net}.  This is a special case of the more general notion of $(t,m,s)$-net defined by Niederreiter in \cite{niederreiter}.  These nets are especially well-distributed finite subsets of $[0,1]^s$ which have proven to be widely useful for quasirandom number generation and Monte Carlo integration, for instance in mathematical finance (~\cite{boyle1997monte}.)

It is not hard to show that the set of $(0,m,2)$-nets (up to the obvious notion of equivalence, where a net is labeled by the set of $(1/2^m) \times (1/2^m)$ squares it intersects) is in bijection with $\dwd{m}$.  In this context, the fact that there are exactly $2^{m2^{m-1}}$ inequivalent $(0,m,2)$-nets (part of Proposition ~\ref{prop: phi is iso}) is proved in~\cite{xiao}.

As in Remark \ref{rem: x and y as permutation matrices}, we may think of the affine linear group $\mathrm{GL}(\mathbb{F}_2^m) \ltimes \mathbb{F}_2^m$ as a subgroup of $S_{2^m}$.  Write $B$ for the subgroup of upper triangular matrices in $\mathrm{GL}(\mathbb{F}_2^m)$ and $G$ for $B \ltimes \mathbb{F}_2^m$.  Then $G$ permutes the basic intervals in $[2^m]$, which means that $\dwd{m}$ is preserved under both left and right multiplication by $G$.  In particular, this means that the double coset $B x_m B$ is contained in $\dwd{m}$.  Note that $x_m$, considered as an element of  $\mathrm{GL}(\mathbb{F}_2^m)$, is a permutation matrix, corresponding to the reversal in $S_m$; so $B x_m B$ is the large cell in the Bruhat decomposition of $\mathrm{GL}(\mathbb{F}_2^m)$.  This very natural construction was also known in the low-discrepancy sequence setting: the $(0,m,2)$-nets corresponding to $B x_m B$ are called the {\em digital} $(0,m,2)$-nets. Indeed, one finds in that literature the condition that a matrix has all principal minors nonsingular (see for instance Corollary 4.54 of \cite{dick2010digital}) but the fact that this condition picks out the large Bruhat cell does not seem to have been noticed, just as the fact that the $(0,m,2)$-nets form a Bruhat interval (as we show in Theorem~\ref{thm: main}) has not been noticed.  It would be interesting to mine the literature on low-discrepancy sequences to see if there are yet more constructions there which have unexpected connections to algebraic combinatorics.

\subsection{Powers of permutahedra}
\label{rmk: higher-t}
   The arguments above generalize as follows. For each $t, m > 0$, there is a set of $t$-adically well-distributed permutations $\dwd{m}^{(t)} \subseteq S_{t^m}$, defined by imitating Definition \ref{def: dwd}. Note, in particular, that $\dwd{1}^{(t)} = S_t$. Further examples of these permutations are given by  $x_m^{(t)}, y_m^{(t)} \in S_{t^{m}}$, which are the affine transformations of $(\Z/t\Z)^m$ given by
    \[
    x_m^{(t)}(a_0, \dots, a_{m-1}) = (a_{m-1}, \dots, a_0); \qquad y_m^{(t)}(a_0, \dots, a_{m-1}) = (t-1-a_{m-1}, \dots, t-1-a_0).
    \]
    Note that $x_m^{(t)}$ and $y_m^{(t)}$ are involutions, and $x_m^{(t)}y_m^{(t)} = y_m^{(t)}x_m^{(t)}$ is the longest element in $S_{t^m}$. Definition \ref{def: order-iso} generalizes to give a map $\varphi: \dwd{m}^{(t)} \to S_t^{mt^{m-1}}$. Using this one can show that
    \[
    [x_m^{(t)}, y_m^{(t)}] = \dwd{m}^{(t)} \cong S_t^{mt^{m-1}},
    \]
    where the last isomorphism is that of posets with the Bruhat order.

    For example, to construct a $3$-adically well-distributed permutation in $S_9$ we start with the following configuration:

\begin{center}
\begin{tikzpicture}[scale=0.4]
\draw[step=1.0, color=lightgray] (0,0) grid (9,9);

\draw[step=3.0, very thick] (0,0) grid (9,9);

\node at (1.5,1.5) {\Huge 1};
\node at (4.5,1.5) {\Huge 1};
\node at (7.5,1.5) {\Huge 1};

\node at (1.5,4.5) {\Huge 1};
\node at (4.5,4.5) {\Huge 1};
\node at (7.5,4.5) {\Huge 1};

\node at (1.5,7.5) {\Huge 1};
\node at (4.5,7.5) {\Huge 1};
\node at (7.5,7.5) {\Huge 1};

\node at (1.5, -0.5) { \color{blue} $\sigma_1$};
\node at (4.5, -0.5) { \color{blue} $\sigma_2$};
\node at (7.5, -0.5) { \color{blue} $\sigma_3$};

\node at (9.75, 1.5) { \color{red} $\tau_3$};
\node at (9.75, 4.5) { \color{red} $\tau_2$};
\node at (9.75, 7.5) { \color{red} $\tau_1$};
\end{tikzpicture}
\end{center}

\noindent where the large $1$'s indicate that there is a unique entry equal to $1$ in the corresponding $3 \times 3$-block, and $\sigma_{i}, \tau_{i} \in S_3$ are permutations that indicate the relative positions of the $1$'s in the corresponding column and row, respectively. This gives us the bijection between the $3$-adically well-distributed permutations in $S_9$ and $S_3^{6}$. To give a more concrete example:
\begin{center}
\begin{tikzpicture}[scale=0.5]
\draw[step=1.0, color=lightgray] (0,0) grid (9,9);

\draw[step=3.0, very thick] (0,0) grid (9,9);

\node at (1.5, 7.5) {1};
\node at (3.5, 6.5) {1};
\node at (6.5, 8.5) {1};

\node at (0.5, 4.5) {1};
\node at (5.5, 5.5) {1};
\node at (7.5, 3.5) {1};

\node at (2.5, 0.5) {1};
\node at (4.5, 1.5) {1};
\node at (8.5, 2.5) {1};

\node at (1.5, -0.5) { \color{blue} $[2,1,3]$};
\node at (4.5, -0.5) { \color{blue} $[1,3,2]$};
\node at (7.5, -0.5) { \color{blue} $[1,2,3]$};

\node at (10.2, 1.5) { \color{red} $[3,2,1]$};
\node at (10.2, 4.5) { \color{red} $[2,1,3]$};
\node at (10.2, 7.5) { \color{red} $[2,3,1]$};
\end{tikzpicture}
\end{center}

Changing the last column from the identity $[1,2,3]$ to $[2,1,3]$ results in the following matrix. 

\begin{center}
\begin{tikzpicture}[scale=0.5]
\draw[step=1.0, color=lightgray] (0,0) grid (9,9);

\draw[step=3.0, very thick] (0,0) grid (9,9);

\node at (1.5, 7.5) {1};
\node at (3.5, 6.5) {1};
\node at (7.5, 8.5) {1};

\node at (0.5, 4.5) {1};
\node at (5.5, 5.5) {1};
\node at (6.5, 3.5) {1};

\node at (2.5, 0.5) {1};
\node at (4.5, 1.5) {1};
\node at (8.5, 2.5) {1};

\node at (1.5, -0.5) { \color{blue} $[2,1,3]$};
\node at (4.5, -0.5) { \color{blue} $[1,3,2]$};
\node at (7.5, -0.5) { \color{blue} $[2,1,3]$};

\node at (10.2, 1.5) { \color{red} $[3,2,1]$};
\node at (10.2, 4.5) { \color{red} $[2,1,3]$};
\node at (10.2, 7.5) { \color{red} $[2,3,1]$};
\end{tikzpicture}
\end{center}
and we see that we have a covering relation in the Bruhat order.

\begin{remark}
    Note that a usual Sudoku puzzle is a collection $(\sigma_1, \dots, \sigma_9)$ of $3$-adically well-distributed permutations in $S_9$, satisfying the condition that $\sigma_i\sigma_j^{-1}$ does not have fixed points if $i \neq j$. From this viewpoint, we can define generalized Sudoku puzzles as a collection of $t^m$ $t$-adically well-distributed permutations $(\sigma_1, \dots, \sigma_{t^{m}})$ in $S_{t^{m}}$ such that $\sigma_i\sigma_j^{-1}$ does not have fixed points if $i \neq j$. For example, the following is a $2$-adic Sudoku puzzle of elements in $S_8$. 
    
\begin{center}
    \begin{tikzpicture}[scale=0.5]
\foreach \x in {0,...,8}
\draw (0, \x) -- (8, \x);
\foreach \x in {0,...,8}
\draw (\x, 0) to (\x, 8);



\node at (0.5, 0.5) {\small 7};
\node at (1.5, 0.5) {\small 8};
\node at (2.5, 0.5) {\small 2};
\node at (3.5, 0.5) {\small 3};
\node at (4.5, 0.5) {\small 4};
\node at (5.5, 0.5) {\small 1};
\node at (6.5, 0.5) {\small 6};
\node at (7.5, 0.5) {\small 5};

\node at (0.5, 1.5) {\small 6};
\node at (1.5, 1.5) {\small 5};
\node at (2.5, 1.5) {\small 4};
\node at (3.5, 1.5) {\small 1};
\node at (4.5, 1.5) {\small 7};
\node at (5.5, 1.5) {\small 8};
\node at (6.5, 1.5) {\small 3};
\node at (7.5, 1.5) {\small 2};

\node at (0.5, 2.5) {\small 3};
\node at (1.5, 2.5) {\small 2};
\node at (2.5, 2.5) {\small 7};
\node at (3.5, 2.5) {\small 8};
\node at (4.5, 2.5) {\small 6};
\node at (5.5, 2.5) {\small 5};
\node at (6.5, 2.5) {\small 4};
\node at (7.5, 2.5) {\small 1};

\node at (0.5, 3.5) {\small 4};
\node at (1.5, 3.5) {\small 1};
\node at (2.5, 3.5) {\small 6};
\node at (3.5, 3.5) {\small 5};
\node at (4.5, 3.5) {\small 3};
\node at (5.5, 3.5) {\small 2};
\node at (6.5, 3.5) {\small 7};
\node at (7.5, 3.5) {\small 8};

\node at (0.5, 4.5) {\small 8};
\node at (1.5, 4.5) {\small 7};
\node at (2.5, 4.5) {\small 3};
\node at (3.5, 4.5) {\small 2};
\node at (4.5, 4.5) {\small 1};
\node at (5.5, 4.5) {\small 4};
\node at (6.5, 4.5) {\small 5};
\node at (7.5, 4.5) {\small 6};

\node at (0.5, 5.5) {\small 5};
\node at (1.5, 5.5) {\small 6};
\node at (2.5, 5.5) {\small 1};
\node at (3.5, 5.5) {\small 4};
\node at (4.5, 5.5) {\small 8};
\node at (5.5, 5.5) {\small 7};
\node at (6.5, 5.5) {\small 2};
\node at (7.5, 5.5) {\small 3};

\node at (0.5, 6.5) {\small 2};
\node at (1.5, 6.5) {\small 3};
\node at (2.5, 6.5) {\small 8};
\node at (3.5, 6.5) {\small 7};
\node at (4.5, 6.5) {\small 5};
\node at (5.5, 6.5) {\small 6};
\node at (6.5, 6.5) {\small 1};
\node at (7.5, 6.5) {\small 4};

\node at (0.5, 7.5) {\small 1};
\node at (1.5, 7.5) {\small 4};
\node at (2.5, 7.5) {\small 5};
\node at (3.5, 7.5) {\small 6};
\node at (4.5, 7.5) {\small 2};
\node at (5.5, 7.5) {\small 3};
\node at (6.5, 7.5) {\small 8};
\node at (7.5, 7.5) {\small 7};
\end{tikzpicture}
\end{center}
\end{remark}

\subsection{Kazhdan-Lusztig $d$-invariant}
\label{sec:R polynomials}

We define the $d$-Kazhdan-Lusztig invariant mentioned in the introduction.
It arises as the absolute value of the coefficient of the second largest power in a Kazhdan-Lusztig $R$-polynomial, which we recall its definition now for permutations although they can be defined for arbitrary Coxeter systems. 

\begin{definition}
There is a unique family of polynomials
$\{R_{u,v}(q)\}_{u,v\in S_n}\subseteq \mathbb{Z}[q]$
satisfying the following conditions:
\begin{enumerate}
\item $R_{u,v}(q)=0$ if $u\not\le v$;
\item $R_{u,v}(q)=1$ if $u=v$;
\item For $1\leq i <n$, let $s_{i}= (i ,\,i+1)$.

If $\ell (vs_i) <\ell(v)$ then
\begin{equation}
R_{u,v}(q)=
\begin{cases}
R_{us_i,vs_i}(q), & \text{if } vs_i <u,\\[2pt]
q R_{us_i,vs_i}(q)+(q-1)R_{u,vs_i}(q), & \text{if } vs_i >v.
\end{cases}
\end{equation}
\end{enumerate}
\end{definition}
We can now define the $d$-invariant.
\begin{definition}
For $x, y \in S_n$ we introduce $d_{x,y}\in \mathbb{Z}$ by
\begin{equation}
R_{x,y}(q) = q^{\ell(y)-\ell(x)} - d_{x,y}, q^{\ell(y)-\ell(x)-1} + \text{ lower degree terms}.
\end{equation}
\end{definition}
We have the following recursion to compute $d_{x,y}$.
\begin{enumerate}
\item $d_{x,x}=0$ for any $x\in W$;
\item for any $x\le y$ and any simple transposition $s_i$ such that $y s_i<y$ we have
\end{enumerate}
\begin{equation}\label{eq: computing d-invariant}
d_{x,y}=
\begin{cases}
d_{xs_i,ys_i}, & \text{if } x s_i<x,\\[2pt]
d_{x,y s_i}+1, & \text{if } x s_i>x \text{ and } x s_i \nleq y s_i,\\[2pt]
d_{x,y s_i}, & \text{if } x s_i>x \text{ and } x s_i \le y s_i.
\end{cases}
\end{equation}
\begin{cor} \label{cor: d-invariant is superlinear}
Let $m\geq 1$.
Let $x,y\in [x_m,y_{m}]$ be such that $x\leq y$. 
Then, we have
\begin{equation} \label{eq: R when the interval is a cube}
R_{x,y}(q)=(q-1)^{\ell(y)-\ell(x)}.
\end{equation}
In particular, $R_{x_m,y_m}(q)= (q-1)^{m2^{m-1}}$ and $d_{x_m,y_m}=m2^{m-1}$.
\end{cor}
\begin{proof}
Theorem \ref{thm: main} implies that the Bruhat interval $[x,y]$ is a hypercube.
In particular, $[x,y]$ does not contain any interval isomorphic to $S_3$.
Therefore, \eqref{eq: R when the interval is a cube} follows by a direct application of \cite[Theorem 6.3]{brenti1994combinatorial}.
Finally, by Lemma \ref{xdwd} we have $\ell(y_m)-\ell(x_m)=m2^{m-1}$, and the last claim follows.
\end{proof}

We use the name $d$-invariant because of the work of Patimo \cite{patimo2021combinatorial}, who showed that this value is a combinatorial invariant of Bruhat intervals. 
Namely, it depends only on the unlabeled Hasse diagram of the interval and not on its endpoints. 
Building on this fact, he proved the Combinatorial Invariance Conjecture of Lusztig and Dyer for the $q$-coefficient of a Kazhdan--Lusztig polynomial.

The main ingredient in his proof is a combinatorial formula for computing $d_{x,y}$, namely
\begin{equation} \label{eq: d-combinatorial}
    d_{x,y}= \min\{\, |F| \mid F \subset E_{[x,y]} \text{ and } F^{\diamond} = E_{[x,y]} \,\},
\end{equation}
where $E_{[x,y]}$ denotes the set of edges of the Hasse diagram of $[x,y]$, and $F^{\diamond}$ denotes the diamond-closure of $F$. For the precise definition of $F^{\diamond}$ we refer to \cite[Definition 4.5]{patimo2021combinatorial}. Roughly speaking, $F^{\diamond}\supset F $ is obtained as follows: whenever two adjacent edges of a diamond (i.e., a rank-two hypercube) already lie in $F^{\diamond}$, we add to $F^{\diamond}$ the remaining two adjacent edges of that diamond. We repeat this process until no further edges can be added.

In general, computing $d_{x,y}$ using \eqref{eq: d-combinatorial} is a difficult task. However, when $[x,y]$ is a hypercube, it is straightforward to verify that a minimal set $F$ satisfying the condition in \eqref{eq: d-combinatorial} is given by the set of edges incident to $x$. In this way we recover the formula $d_{x_m,y_m}=m2^{m-1}$ directly from the combinatorial definition of the $d$-invariant.

\subsection{Cluster varieties}

We say that an affine algebraic variety $X$ is a cluster variety if its coordinate algebra $\C[X]$ admits the structure of a cluster algebra, cf. \cite{FZ}.  This means that $X$ can be covered, up to codimension $2$, with algebraic tori, and each of these tori admits a coordinate system consisting of regular functions on $\C[X]$ (the \emph{cluster variables}) such that the transition functions between the tori can be codified into a combinatorial rule known as \emph{cluster mutation}. The cluster variables can be separated into \emph{mutable} and \emph{frozen}. The frozen variables are precisely those that are units in $\C[X]$, and appear as coordinates in every algebraic torus of the cluster structure on $X$. While an algebraic variety may admit many non-equivalent cluster structures, the number of frozen variables is an invariant of $X$, cf. \cite{GLS-factorial}. The existence of a cluster structure has many consequences for the geometry of $X$, for example, it implies the existence of a $\Z$-form with positive structure constants, \cite{GHKK}. 

Recently, it has been shown that open Richardson varieties in the flag variety admit cluster structures in \cite{CGGLSS} and, independently, in \cite{GLSBII}. By the main result of \cite{CGGSSBS}, the cluster structures constructed in both works coincide. The work \cite{GLSBII} uses Deodhar geometry \cite{Deodhar} to construct the cluster structure. In particular, the number of \emph{mutable} variables in a seed coincides with the number of codimension-$1$ Deodhar strata in the Richardson variety, and the total number of cluster variables is $\dim R^{\circ}_{x,y} = \ell(y) - \ell(x)$. Denoting the number of frozen variables by $f_{x,y}$ we have
\begin{equation}\label{eq: frozen-variables}
f_{x,y} = \ell(y) - \ell(x) - \#\mathfrak{D}^{1}_{x,y}
\end{equation}
where $\mathfrak{D}^{1}_{x,y}$ is the set of codimension-$1$ Deodhar strata. 

On the other hand, Deodhar in \cite[Theorem 1.3]{Deodhar} gives an expression for the $R$-polynomial in terms of the Deodhar stratification:
\[
R_{x,y}(q) = \sum_{\mathfrak{s}} (q-1)^{a(\mathfrak{s})}q^{b(\mathfrak{s})}.
\]
where $\mathfrak{s}$ runs over all Deodhar strata, and $a(\mathfrak{s}) + 2b(\mathfrak{s}) = \ell(y) - \ell(x)$. The quantity $b(\mathfrak{s})$ is precisely the codimension of the stratum $\mathfrak{s}$. There is a unique stratum $\mathfrak{s}_0$ with $b(\mathfrak{s}_0) = 0$. It follows that we can write
\[
R_{x,y}(q) = (q-1)^{\ell(y) - \ell(x)} + \#\mathfrak{D}^{1}_{x,y}(q-1)^{\ell(y) - \ell(x) -2}q + \mbox{ lower terms.}
\]
We obtain
\begin{equation}\label{eq: R-polynomial-Deodhar}
R_{x,y}(q) = q^{\ell(y) - \ell(x)} - (\ell(y) - \ell(x) - \#\mathfrak{D}^1_{x,y})q^{\ell(y) - \ell(x) - 1} + \mbox{ lower terms.}
\end{equation}
Comparing \eqref{eq: frozen-variables} and \eqref{eq: R-polynomial-Deodhar} we conclude that $f_{x,y} = d_{x,y}$, i.e., the number of frozen variables in the Richardson variety $R^{\circ}_{x,y}$ is precisely the $d$-invariant $d_{x,y}$.

Any cluster variety admits a faithful action of an algebraic abelian group, and if the cluster variety has \emph{really full rank}, this abelian group is a torus of rank equal to the number of frozen variables, see \cite[Section 5]{LamSpeyer} for details. In particular, open Richardson varieties have really full rank, so we have a faithful action
\begin{equation}\label{eq: cluster-action}
(\C^{\times})^{d_{x,y}} \curvearrowright R^{\circ}_{x,y}.
\end{equation}

One problem with the cluster-theoretic construction of this action is that it is non-explicit and its construction depends on intricate combinatorics (see e.g. \cite[Example 4.6]{Kim}) , so it would be desirable to have a more explicit form of it. Identifying $(\C^{\times})^{n-1}$ with the torus of diagonal matrices in $\mathrm{PGL}(n)$, it is easy to obtain a natural action of $(\C^{\times})^{n-1}$ on $R^{\circ}_{x,y}$. In many cases (for example, when $x = e$ (\cite[Section 4]{CGSS}) or when $y$ is a Grassmannian permutation (\cite{GL-positroid})), we have $d_{x,y} \leq n-1$ and the action \eqref{eq: cluster-action} is a quotient of the natural $(\C^{\times})^{n-1}$-action. When $d_{x,y} > n-1$, it is not clear how to obtain the action of the remaining $\C^{\times}$-factors in \eqref{eq: cluster-action}. In particular, it would be very interesting to obtain an explicit action of $(\C^{\times})^{\ell(y_n) - \ell(x_n)}$ on $R^{\circ}_{x_n, y_n}$ without using that $R^{\circ}_{x_n, y_n}$ is already isomorphic to a torus. Moreover, it would be interesting to know if the $(\C^{\times})^{\ell(y_n) - \ell(x_n)}$-action on the \emph{open} Richardson variety $R^{\circ}_{x_n,y_n}$ extends to the \emph{closed} Richardson variety $\overline{R}_{x_n,y_n}$. We remark that Theorem \ref{thm: main} implies, using \cite[Corollary 3.2]{Abe-Billey} and \cite[Corollary 1.3]{Knutson-Woo-Yong}, that $\overline{R}_{x_n, y_n}$ is smooth.

\subsection{Moduli of Bruhat interval embeddings}

The Bruhat graph of an interval $[x,y]$ has vertices consisting of the permutations in the interval and a directed edge $u \to v$ if $u \le v$ and $u =tv$ for some reflection $t = (i,j) \in S_n$, in which case the edge is labelled by $t$. One can recover the unlabelled Bruhat graph of an interval from the poset \cite{Dyer-BG}.

Bruhat graphs have natural embeddings. The symmetric group $S_n$ acts on $\mathbb{R}^n$ via permutation of coordinates. Let us fix $\rho = (0, 1, \dots, n-1)$ and send $u \in [x,y]$ to $u(\rho)$ and $u \to v$ to the straight line connecting $u(\rho)$ to $v(v)$. This embedding has a geometric meaning: it is the image of the $0-$ and $1$-dimensional torus orbits in the closed Richardson variety under the moment map  \cite{BradenMacPherson}. We will refer to this embedding as the \emph{geometric} embedding below, and stress that it is extra information.

 For several questions in combinatorics and Kazhdan-Lusztig theory, it is more natural to consider the Bruhat graph. One such example, is the combinatorial invariance conjecture of Dyer and Lusztig, which asserts that isomorphic intervals give rise to the same Kazhdan-Lusztig polynomial. It is known \cite{Dyer-BG, BradenMacPherson} that one can recover the Kazhdan-Lusztig polynomial from the Bruhat graph \emph{together} with its embedding.
 A crucial difficulty is that although two intervals may be isomorphic, their geometric embeddings might be very different. For example, the natural ambient spaces in which the embeddings live might not even be of the same dimension. This motivated Hone, Klein and the last author to consider a general class of embeddings of Bruhat intervals, of which the geoemtric embeddings discussed above are an example \cite{HoneKleinWilliamson}.

In order to explain this definition, we need to briefly recall reflection subgroups. In our setting of the symmetric group, these are subgroups generated by any subset of transpositions. Any such subgroup is isomorphic to a product of (possibly smaller) symmetric groups. We say that a reflection subgroup is of rank 2 if it is isomorphic to $S_2 \times S_2$ or $S_3$. One can detect the intersections of cosets of rank 2 reflection subgroups with any Bruhat interval. These are precisely the diamond complete subgraphs isomorphic to an arrow, diamond or a full $S_3$ subgraph (see Figure \ref{fig:rank2subgraphs}). We will refer to these as \emph{rank 2 subgraphs}.

The following definition (central to \cite{HoneKleinWilliamson}) attempts to abstract the  properties of the embeddings of Bruhat graphs that occur in nature:

\begin{definition}\label{def: embedding}
    Let $x, y \in S_n$. An embedding of the Bruhat graph $\phi : [x,y] \to \mathbb{R}^m$ is \emph{good} if:
    \begin{enumerate}
    \item The images of the vertices lie on a sphere $S^{m-1} \subset \mathbb{R}^m$;
    \item For all rank 2 subgraphs, the images of its points lie on a plane.
    \end{enumerate}
    \end{definition}

It is easy to see that the geometric embeddings considered above satisfy these properties. However, even simple Bruhat intervals may have many more good embeddings than occur in nature. For example, only a few geometric embeddings of diamonds occur in nature, whereas the moduli space of good embeddings inside $\mathbb{R}^2$ is an open set inside the moduli space of 4 points on a circle $S^1 \subset \mathbb{R}^2$. More generally, in the symmetric group the angles between edges in geometric embeddings always lie in the set $\pm \pi/3, \pm \pi/2$ and $\pm 2\pi/3$, whereas no such restrictions are made for good embeddings. Good embeddings appear to provide a natural larger space in which one could attempt to relate the geometric embeddings of isomorphic intervals.

For a fixed interval $[x,y]$, it is natural to ask: what is the largest possible dimension in which this interval admits a good embedding? (We always assume embeddings are such that their vertices are not contained in an affine hyperplane.) This quantity shares at least two  properties with the number of frozen variables discussed earlier: one might (falsely) guess that it is always bounded by $n-1$; and, it is trivially bounded by $\ell(y)  - \ell(x)$. It is plausible that these two numbers are in fact equal, however this is not known.
The setting simplifies dramatically for hypercubes. In this case this largest possible dimension is easily seen to be the dimension of the hypercube, and the moduli of all embeddings can be understood explicitly.

\begin{figure}
\[
\begin{tikzpicture}[xscale=1.4,yscale=0.9,>=stealth]
  \tikzstyle{vertex}=[circle, draw, fill=red!10, inner sep=2pt]

  \node[vertex] (B) at (-8,-1)  {};
  \node[vertex] (T) at (-8,0.5)   {};

  \draw[->,gray!70] (B) -- (T);

  \node[vertex] (LL) at (-7,0.5)  {};
  \node[vertex] (RL) at (-5,0.5)   {};
  \node[vertex] (T)  at (-6,1.5)   {};

  \draw[->,gray!70] (LL) -- (T);
  \draw[->,gray!70] (RL) -- (T);

  \node[vertex] (LL) at (-7,-1)  {};
  \node[vertex] (RL) at (-5,-1)   {};
  \node[vertex] (B)  at (-6,-2)   {};

  \draw[->,gray!70] (B)  -- (LL);
  \draw[->,gray!70] (B)  -- (RL);

  \node[vertex] (RU) at (-3, 1)   {};
  \node[vertex] (LL) at (-4,-1)  {};
  \node[vertex] (RL) at (-2,-1)   {};
  \node[vertex] (B)  at (-3,-2)   {};

  \draw[->,gray!70] (B)  -- (LL);
  \draw[->,gray!70] (B)  -- (RL);
  \draw[->,gray!70] (RL) -- (RU);
  \draw[->,gray!70] (LL) -- (RU);

  \node[vertex] (T)  at (0, 2)   {};
  \node[vertex] (LU) at (-1, 1)  {};
  \node[vertex] (RU) at (1, 1)   {};
  \node[vertex] (LL) at (-1,-1)  {};
  \node[vertex] (RL) at (1,-1)   {};
  \node[vertex] (B)  at (0,-2)   {};

  \draw[->,gray!70] (B)  -- (LL);
  \draw[->,gray!70] (B)  -- (RL);
  \draw[->,gray!70] (LL) -- (LU);
  \draw[->,gray!70] (RL) -- (RU);
  \draw[->,gray!70] (LU) -- (T);
  \draw[->,gray!70] (RU) -- (T);

  \draw[->,gray!70] (B)  -- (T);
  \draw[->,gray!70] (LL) -- (RU);
  \draw[->,gray!70] (RL) -- (LU);
\end{tikzpicture}
\]
\caption{Rank two subgraphs of Bruhat intervals for symmetric groups: the arrow, incomplete diamonds, diamond and full $S_3$.}
\label{fig:rank2subgraphs}
\end{figure}
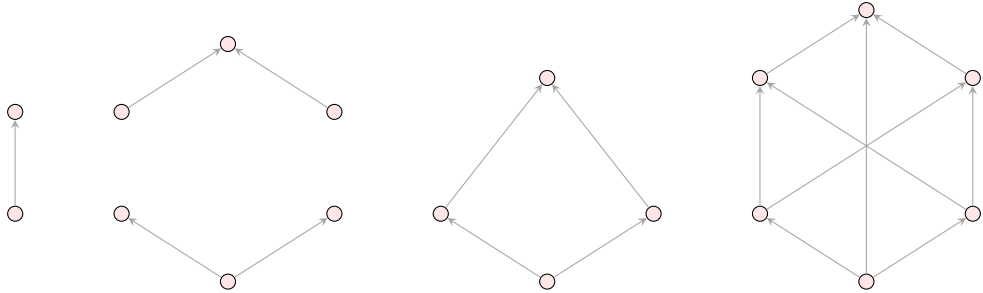

\section{Experiments with FunSearch and AlphaEvolve}
\label{sec:experiments}

The main theorems of the present paper were suggested to us by interesting examples that were discovered computationally.  
Since the computational methods employed here (genetic algorithms using large language models as a mode of reproduction) are novel, we will devote slightly more time than is customary to describing the course of experimentation in a fair amount of narrative detail, in the hope that this discussion will be useful to other mathematicians interested in employing these methods.

The initial problem we considered was that of finding pairs of permutations $x,y$ in $S_n$ whose $d$-invariant $d_{x,y}$ is as large as possible. 

This problem has two features which, together, have tended to indicate promising contexts for exploration using machine learning methods:
\begin{itemize}
\item The search space (pairs of permutations) has size $(n!)^2$, too large to search exhaustively for any but the very smallest $n$.
\item The function $d_{x,y}$ to be maximized can be computed quickly and reliably using \eqref{eq: computing d-invariant}. 
\end{itemize}

Other problems with these features that have been subjects for machine learning experiments of this kind include:
\begin{enumerate}
    \item {\bf{The cap set problem}}. 
    
    Finding large subsets of $(\Z/3\Z)^n$ with no three-term arithmetic progressions.
    \item {\bf{No-Isosceles Problem}}.
    
    Finding large subsets of the $n \times n$ grid with no isosceles triangles.
    \item {\bf{The extremal number of $C_4$}}.
    
    Finding graphs on $n$ vertices with as many edges as possible and no $4$-cycles.
\end{enumerate}

 Note that the search spaces for these problems have size exponential in $n^2$ or even doubly exponential. 
 The search space treated here is, in some sense, relatively small among those that are too big. 
 This may help explain the relatively successful performance observed in the experiments reported here.

Let $f: \mathbb{N} \rightarrow \mathbb{N}$ be the function given by
\begin{equation}
    f(n) = \max \{  d_{x,y }  \mid x,y\in S_n \}.
\end{equation}
It is quite easy to show that $f(n)$ is at least within a constant of $n$. 
For instance, we have $d_{e,w_0}=n-1$. 
Exhaustive computation for $n \leq 7$ yields

\begin{table}[h]
    \centering
    \begin{tabular}{|c|c|c|c|c|c|c|c|} \hline
$n$     & 1 & 2 & 3 & 4 & 5 & 6 & 7 \\
\hline
$f(n)$  &0  & 1 & 2 & 4 & 5 & 7 & 9 \\ \hline
    \end{tabular}
    \caption{Small values of $f(n)$}
    \label{tab:values f(n)}
\end{table}

We found a pair of permutations in $S_8$ with $d$-invariant $12$, but were not sure whether this value was maximal. 
Indeed, for each $n$ we were able to construct a pair of permutations $(x,y)$ such that $d_{x,y} = 2n-4$ if $n$ is even and $2n-5$ if $n$ is odd.  
These are formed by a construction that is easy to describe:
for instance, when $n=12$, a pair of permutations with $d_{x,y} = 20$ is given by 
\begin{equation}
        x=  (2,5)(4,7)(6,9)(8,11)   \qquad \mbox{and} \qquad   y =(1,12)(2,4)(3,6)(5,8)(7,10)(9,11).
\end{equation}

Our goal as we began our experiments was to understand whether this lower bound was best possible.  As we record in \ref{cor: d-invariant is superlinear}, the answer turned out to be no.

\subsection{FunSearch experiments}

FunSearch is a protocol developed by Google DeepMind in 2023 and described in \cite{funsearch}. 
FunSearch is open-source and various versions are available.
We used the implementation developed by Christofero Fraser-Taliente, which we have found particularly mathematician-friendly, and which is publicly available at \url{https://github.com/kitft/funsearch}.

We recall the main idea.  Our goal is to find two permutations $x,y$ such that $d_{x,y}$ is large.  FunSearch, rather than searching for such a pair directly, searches for a Python program whose output is such a pair.

The main loop in the mechanism looks something like this.  Here, $M$ and $m$ are constants which matter only insofar as $M \gg m$.
\begin{enumerate}
\item 
Suppose given a large number $M$ of Python programs. \label{initial}
\item Run them all.  Discard all those whose output is not a pair of permutations.  Among those that remain, choose the $m$ programs whose output has the largest value of $d_{x,y}$.
\item Send various subsets of these $m$ as input to a large language model, together with a prompt asking the model to output $M$ programs which resemble the ones in the input.
\item Return to step \ref{initial}.
\end{enumerate}

At the initial step, the $M$ programs can be taken to be $M$ copies of an initial program chosen by the user.  In these experiments, we typically initialized with the program that takes an integer $n$ as input and outputs two copies of the identity permutation in $S_n$.

The permutation-producing programs that evolve as FunSearch runs are, for historical reasons, called \texttt{priority}.  
Importantly, the evaluation function (in this case, the function that computes $d_{x,y}$ given input permutations $x$ and $y$) is frozen in place; we supply it to FunSearch at the beginning and it does not evolve. 
This gives the user quite a bit of power and flexibility in guiding the evolution of \texttt{priority} towards whatever objective is desired.

There are many ways to express a permutation in python.  
For these experiments, we chose to render a permutation as a list of integers between $1$ and $n-1$ inclusive; this list is interpreted by the evaluator as a product of simple reflections $s_i$.
In particular, the identity is expressed as the empty list, so our initial program simply outputs two copies of the empty list. 

One might ask how we constrain \texttt{priority} to give output of this form as it evolves under FunSearch.  
The answer is, we can't!
Not directly, at least.  
What the user controls directly is the frozen evaluator, which we can design to give a low score to any \texttt{priority} whose output isn't in the desired form.
In practice (in this experiment and many others) we have found that this practice is sufficient to keep most of the programs in the population producing well-formed output.

Our experiments used the prompt
{\small
\begin{lstlisting}[breaklines=true]
The program is trying to learn a pair of permutations with large d-invariant.

On every iteration, improve priority_v1 over the priority_vX methods from previous iterations.
Make only small changes.
Try to make the code short.
Please do *not* use any randomness.
\end{lstlisting}
}

The request not to use randomness is because when a program with non-deterministic output achieves a high score on a test run, we may not be able to reproduce that performance on later runs.  It's simpler to keep track of what's going on if a program that pleases the evaluator is actually deterministically good.

Our runs primarily used Mistral's model \texttt{mistral-small-3.1-24b-instruct}, which was chosen for cheapness.  
In \cite{generative} we found that, in FunSearch applications, there was no consistent performance benefit to using more expensive models.  
In the present project we occasionally ran FunSearch with input from larger models, but anecdotally observed no improvement in performance.  
The total cost of the FunSearch experiments reported here was roughly 10USD in API calls to Mistral.  
By contrast, \cite{AE} reports that AlphaEvolve did exhibit a benefit when supplied with calls to a more expensive LLM.

In general, programs trained by FunSearch for a fixed instance or a fixed set of instances of a problem tend not generalize well to further instances.
For instance, a program trained to study the No-Isosceles Problem for a  $9 \times 9$ grid does not perform very well when run on a larger grid~\cite{generative}. 
This problem is quite different. 
FunSearch experiments rather quickly (within a half hour) arrived at programs which gave $d_{x,y} = 2n - c$ for many $n$, even when trained on only one $n$.  For example, a run for $n=11$ yielded the following program, which produces a $d$-invariant of $17$.

{\small
\begin{lstlisting}[language=Python, breaklines=true]
def priority(n: int) -> list:
 #Returns a pair of lists of integers between 1 and n-1, to be interpreted as a sequence of adjacent transpositions.
 # n is an int.  We do not constrain the length of the lists.

    a = list(range(1, n-1, 2))[::-1] + list(range(2, n-1, 2))[::-1]
    b = list(range(0, n-2, 2)) + list(range(n//3, n-1, 2))[::-1]
    c = list(range(1, n-1, 3))[::-1] + list(range(2, n-1, 3))[::-1] + list(range(3, n-1, 3))[::-1]
    return [a, b, c]
\end{lstlisting}
}

One amusing aspect is that this program actually produces a list of {\em three} permutations, rather than two. 
It turns out that the evaluator as we wrote it simply discards any extra permutations passed to it as input. 
It is quite common in both FunSearch and AlphaEvolve that LLM-generated programs contain code irrelevant to the task.

The program above produces a pair of permutations with d-invariant $11m+6$ for every $n$ of the form $6m+5$ that we tested, and we imagine it likely does so in general.
This is already unusually good performance; one rarely sees programs that generalize to infinitely many problem instances.

On the other hand, when $n=13$ the program gives a pair of permutations with $d$-invariant $13$, a fairly poor performance.  But a FunSearch run {\em initialized} at the above program very rapidly found a program which yielded $d$-invariant $21$; the only change is that the definition of the second permutation changes to
{\small
\begin{lstlisting}[language=Python, breaklines=true]
b = list(range(0, n-1, 2)) + list(range(n//4, n-1, 2))[::-1]
\end{lstlisting}
}

This suggested manually checking which value of the beginning of the second range (\texttt{n // 3} in the first program trained on $n=11$, \texttt{n // 4} in the second trained on $n=13$) gives optimal performance. It turns out that in practice it seems best always to start the range at $3$. 
This manually modified program yields a pair of permutations with $d$-invariant $2n-5$ for all $n$.\footnote{Or so the data suggests; since this turned out not to be optimal, we did not write down a proof that this works for all $n$.}
This performance was equal to that of the best pair of permutations we came up with for odd $n$, and worse by $1$ for even $n$.  
The permutations generated by this program were not identical to the ones we knew, but were in some sense ``in the same spirit" -- in particular, the number of involutions in both sets of  permutations is linear in $n$, meaning that we are very close to the bottom of the Bruhat order.

\subsection{AlphaEvolve experiments}

We then moved on to AlphaEvolve, another evolutionary protocol designed by DeepMind as a successor to FunSearch.  AlphaEvolve is expected to be released as an open-source package in the near future, but for the moment we ran our trials internally at DeepMind.  AlphaEvolve has already been observed to do well at example-finding in a wide variety of problems in both discrete and continuous areas of mathematics, and in some cases to find examples for a single problem instance that generalize to an infinite family of instances~\cite{AE}.  Customary practice in AlphaEvolve runs is to use a less terse, more directive prompt.  (From \cite{AE}:  ``Giving AlphaEvolve an insightful piece of expert advice in the
prompt almost always led to significantly better results.")  Here's the prompt we used:

The prompt begins

{\small
\begin{lstlisting}[breaklines=true]
    Maximizing the d-invariant of Kazhdan-Lusztig polynomials

Act as an expert in computational algebra and combinatorial optimization. Your task is to find pairs of permutations `x` and `y` in the symmetric group S_n that have the largest possible "d-invariant", denoted d(x,y). This invariant is a coefficient of a Kazhdan-Lusztig polynomial.

You need to produce a search function that, given an integer `n`, finds the best pair of permutations of size `n`. A permutation should be represented as a tuple of integers from 1 to `n`. For example, for n=3, a valid pair could be `((1, 3, 2), (3, 2, 1))`.

\end{lstlisting}
}

and ends

{\small
\begin{lstlisting}[breaklines=true]

**Hint:** It is known that the maximum value of d(x,y) for large `n` is at least `2n-5` (for odd `n`) and `2n-4` (for even `n`). Try to find pairs that meet or exceed this bound!

**Second Hint**: While the above prompt tells you to write a search function, what you really want is a general solution. Your program will be tested on some pretty large values of n, where search won't help you. Instead, while in the short term a search might help you, in the long term the only way you can solve this problem is by finding a general construction. This general construction can of course still contain a tiny bit of search in it, but it should contain some good structure based on ideas and understanding, it should NOT be just a general purpose search heuristic. We are here for the long term, so this is what you should focus on. Research maths is a marathon, not a sprint. Now go and find that general construction that solves the problem for all n!

You got this! I believe in you!!!
\end{lstlisting}
}

Rather than training on a single $n$, we evaluated programs by computing the $d$-invariant of their output for all $n$ from $10$ to $50$ and averaging the results.  The idea is to incentivize evolution to produce programs that work for many or even all $n$, rather than being overfitted to the particular $n$ trained on.  This is the approach called ``generalizer mode" in \cite{AE}.  It's worth noting that the ``generalizer" mode of evaluation has been tried in many FunSearch experiment, with quite weak results; one apparent advantage of AlphaEvolve is it seems better able to aggregate information from combined evaluation on many problem instances.

AlphaEvolve ran for several hours, and produced programs which gave $d$-invariants exceeding $2n-5$ for most $n$; to the eye, the performance of these programs still seemed to be linear in $n$ with constant $2$.  Letting AlphaEvolve run overnight, though, resulted in the qualitative improvement we report on in the main body of this paper.  The resulting program produced $d$-invariants which appeared to be (and quickly proved to be) superlinear in $n$.  The program, its output, and the full prompt used to start the evolution process can be found at \url{https://github.com/JSEllenberg/Bruhat-hypercubes/tree/main}.

This run evaluated around 3000-5000 programs; its model calls were to a combination of Gemini 2.5 Flash and Gemini 2.5 Pro, and solicited from those models a number of tokens that would have cost about 70USD at public rates.\footnote{Of course these dollar amounts are quite unstable; we merely want to make the point that the kind of experiments we're talking about don't require computational resources that are infeasible for individual researchers.}

The evolved program is rather long, and contains substantially more non-functional code than the FunSearch examples above.  (Note that the instruction to "keep the code short" is no longer present in the prompt.)  There is a ``simulated annealing" routine which we believe is almost never actually executed, and a search over a large class of recursive constructions which in fact always ends up using one particular recursive construction.  

Denoting by $\alpha(n)$ the $d$-invariant of the permutations in $S_n$ outputted by the program, we observed what appears to be the relation $\alpha(2n) = 2\alpha(n) + n$, and in particular, that $\alpha(2^m) = m 2^{m-1}$ in the range of values tested.

In fact, after inspection of the code we were able to come up with a simple description of what the program (or rather, the small part of the program that was actually doing the work) was doing when $n = 2^m$, and this is what gave rise to the definition of $x_m$ and $y_m$ appearing in the main body of the paper.  In fact, $x_3$ and $y_3$ were the same permutations in $S_8$ we had already found in $S_8$. So in an alternate universe, we might have noticed that these two permutations on $8$ letters had this simple description in terms of length-$3$ binary expansion.  But we did not in fact notice this until alerted to it by the presence of a program that computed this permutation for binary expansions of any length.

As discussed in section~\ref{ss:lowdisc}, these permutations of $2^m$ letters are already known in the world of low-discrepancy sequences.  We searched the logs for relevant keywords to see whether AlphaEvolve was drawing these permutations from the existing literature, but found no references to low-discrepancy sequences.  We did find some references to bit-reversal permutations, and even programs which computed $x_m$ and $y_m$ for $n = 2^m$; but these programs did not end up in the lineage that led to the final successful program!  That's because these particular programs performed very poorly on values of $n$ which were {\em not} powers of $2$.  This scenario illustrates the importance of thinking about the scoring function.  If $s(n)$ is the $d$-invariant of the two permutations a program outputs in $S_n$, our scoring function was $\sum_{n=10}^{50} s(n)$.  But since the maximal $d$-invariant in $S_n$ is monotone in $n$, we could instead have used this knowledge and evaluated our function by $\sum_{n=10}^{50} \max_{i \leq n} s(i)$.  This would have assigned a high score to the first functions to try the bit-reversal permutations for $n=2^m.$  In this case, it didn't matter, since AlphaEvolve eventually arrived at the bit-reversal permutations again, and this time with a program that performed adequately for all $n$ in the range.  

We had some prior intuition that pairs of permutations $x,y$ with large $d$-invariant would have something to do with large hypercubes in the Bruhat order.  Thus, it was very appealing and not completely surprising when we noticed that, for $m=3$ and $m=4$, the Kazhdan-Lusztig R-polynomial is a power of $(q-1)$, which is consistent with the Bruhat interval $[x_m, y_m]$ being a hypercube (that is, equivalent as a directed graph to $(\Z/2\Z)^d$ with edges corresponding to changing a single coordinate from $0$ to $1$.)  It thus seemed likely that this interval was in fact a hypercube of dimension $m 2^{m-1}$.  We expected that this might be difficult to prove, since in general it is challenging even to compute the cardinality of the Bruhat interval between two permutations, let alone determine its structure.  In this case, however, as we explain in the main body of the paper, we were able to find an explicit description of the permutations lying in the Bruhat interval $[x_m, y_m]$ which made everything tractable.

It is probably not a coincidence that the three desirable phenomena
\begin{itemize}
\item Qualitative, not just incremental improvement over previous best known optimum for the combinatorial problem;
\item Highly interpretable program;
\item Output which can be proven to give good examples for all $n$, not just those in the training range, and which seems to be of interest for reasons beyond its good score on the trained-for objective;
\end{itemize}

all occur together in this case. 

It's interesting that we found a very large hypercube while running a search optimized for a different objective, maximizing the $d$-invariant.  It is interesting to wonder whether AlphaEvolve would have done as well if asked directly to find a Bruhat interval that was a large hypercube, rather than a large $d$-invariant.

\bibliographystyle{plain}
\bibliography{main.bib}

\end{document}